%% file: AISTATS-arxiv.tex
\documentclass{article}
\usepackage{amsmath,amsthm,amssymb,graphicx,cancel}
\usepackage[usenames,dvipsnames]{color}
\usepackage{natbib}
\newcommand{\E}{\mathbb{E}}
\newcommand{\R}{\mathbb{R}}
\newcommand{\MMD}{\mathrm{MMD}}
\newtheorem{theorem}{Theorem}
\newtheorem{corollary}{Corollary}
\newtheorem{lemma}{Lemma}
\newtheorem{observation}{Observation}
\usepackage[margin=1.25in]{geometry}

\begin{document}

\title{On the High-dimensional Power of Linear-time Kernel Two-Sample Testing under Mean-difference Alternatives}

\author{\\
Aaditya Ramdas$^{12}$\footnote{Both student authors had equal contribution.}\\
\texttt{aramdas@cs.cmu.edu}\\ 
\and \\
Sashank J. Reddi$^{2*}$\\
\texttt{sjakkamr@cs.cmu.edu}\\
\and \\
Barnab\'{a}s P\'{o}cz\'os$^2$\\
\texttt{bapoczos@cs.cmu.edu} \\
\and \\
Aarti Singh$^2$ \\
\texttt{aarti@cs.cmu.edu} \\ 
\and \\
Larry Wasserman$^{12}$ \\
\texttt{larry@stat.cmu.edu} \\
\and \\
Department of Statistics$^1$ and Machine Learning Department$^2$\\
Carnegie Mellon University\\ \\
}

%

%
%
%
%
%

\maketitle

\begin{abstract}
Nonparametric two sample testing deals with the question of consistently deciding if two distributions are different, given samples from both, without making any parametric assumptions about the form of the distributions. The current literature is split into two kinds of tests - those which are consistent without any assumptions about how the distributions may differ (\textit{general} alternatives), and those which are designed to specifically test easier alternatives, like a difference in means (\textit{mean-shift} alternatives).
 
The main contribution of this paper is to explicitly characterize the power 
of a popular nonparametric two sample test, designed for general alternatives, under a mean-shift alternative in the  high-dimensional setting.
Specifically, we explicitly derive the power of the linear-time Maximum Mean Discrepancy statistic using the Gaussian kernel, where the dimension and sample size can both tend to infinity at any rate, and the two distributions differ in their means. As a corollary, we find that if the signal-to-noise ratio is held constant, then the test's power goes to one if the number of samples increases faster than the dimension increases. This is the first explicit power derivation for a general nonparametric test in the high-dimensional setting, and also the first analysis of how tests designed for general alternatives perform when faced with easier ones.
\end{abstract}

\section{Introduction}

The central topic of this paper is nonparametric two-sample testing, in which we try to detect a difference between two $d$-dimensional distributions $P$ and $Q$ based on $n$ samples from both, i.e. deciding whether two samples are drawn from the same distribution. We will be concerned with the following two settings, the first of which deals with \textit{general} alternatives (GA), i.e.
\begin{equation}\tag{GA}
H_0 : P = Q ~\mbox{~ v.s. ~}~ H_1 : P \neq Q.
\end{equation}
It is called nonparametric two-sample testing because no parametric assumptions are made about the form of $P,Q$ (like Gaussianity or exponential families). We use the term \textit{general} alternatives to mean that the difference between $P,Q$ need not have a simple form. In contrast, the second setting that we are concerned about deals with \textit{mean-shift} alternatives (MSA), i.e.
\begin{equation}\tag{MSA}
H_0 : \mu_P = \mu_Q ~\mbox{~ v.s. ~}~ H_1 : \mu_P \neq \mu_Q
\end{equation}
where $\mu_P = \E_{X \sim P}[X]$ and $\mu_Q = \E_{Y \sim Q}[Y]$. It is still nonparametric two-sample testing, since we make no assumptions about $P,Q$, but deals with \textit{easier} alternatives, meaning that we specify the exact form in which $P$ and $Q$ differ, i.e. they differ in their means. Parametric two-sample testing (for example, when $P,Q$ are Gaussian) is also important, but will be out of the scope of our discussion; see \cite{lopes} for a recent example. We assume equal number  $n$ of samples for simplicity; our results would also go through if $n_1/(n_1+n_2) \rightarrow c \in (0,1)$ as $n_1,n_2 \rightarrow \infty$.

\subsection{Hypothesis testing terminology}

Let $X^{(n)} = \{x_1,...,x_n\} \sim P$ and $Y^{(n)} = \{y_1,...,y_n\} \sim Q$ be the two sets of samples, where $x_i,y_j \sim \R^d$ for all $1 \leq i,j \leq n$. A \textit{test} is any function or algorithm that takes $X^{(n)}, Y^{(n)}$ as input, and outputs $\{0,1\}$ where $1$ is interpreted to mean that it \textit{rejects the null hypothesis $H_0$}, and $0$ is interpreted to mean that \textit{there is insufficient evidence to reject $H_0$}. 
A test is characterized by its false positive rate or type-1 error 
$$\alpha = P( \mbox{rejecting }H_0 ~|~ H_0\mbox{ is true})$$ 
and its false negative rate or type-2 error 
$$\beta= P( \mbox{not rejecting } H_0 ~|~ H_1\mbox{ is true} ).$$
There is usually a tradeoff involved - decreasing one error rate increases the other. Hence, one sometimes fixes $\alpha$ to some small value (say $0.01$), and refers to  $\phi = 1-\beta$ as the \textit{power of the test at $\alpha=0.01$}. A test is classically called consistent if for any fixed $\alpha$, the power $\phi \rightarrow 1$ as $n \rightarrow \infty$ whenever $H_0$ is false.

Many tests in the literature, including the ones we will consider, calculate a test statistic $T$ (as a function of $X^{(n)}, Y^{(n)}$), and reject the null hypothesis if $T > c_\alpha$, where the threshold $c_\alpha$ depends on the distribution of $T$ under $H_0$ and on a pre-defined $\alpha$. See \cite{lehmann06} for a detailed introduction.

\subsection{Motivation}

Our first motivation comes from the fact that there is a big difference between the classical setting of fixing $d$ while letting $n \rightarrow \infty$, and the high-dimensional (HD) setting obtained when
\begin{equation}\tag{HD}
(n,d) \rightarrow \infty
\end{equation} 
A test would be called consistent under HD if for any fixed $\alpha$, the power $\phi \rightarrow 1$ as $(n,d) \rightarrow \infty$ whenever $H_0$ is false. 
It is of vital importance, both theoretically and practically, to understand the power of tests in such settings, and to characterize the rate at which $n$ must grow as a function of $d$ so that the test is still consistent. While classical tests were proposed for the low-dimensional settings, over the past two decades several tests have been proposed specifically for MSA and studied in the HD setting; see Subsection \ref{sec:refmsa}. However, to the best of our knowledge there has been no formal and precise characterization of power of tests designed for GA in high dimensions. 

Our second motivation comes from the observation that there is no literature on how tests designed for GA perform under MSA. In other words, while it is expected that tests designed for MSA will not be consistent against more general GA, it is unclear how exactly tests designed for general alternatives fare when when faced with a mean-shift alternative.

\subsection{Related Work (MSA)}\label{sec:refmsa}

It is well known (see \cite{kariya81,simaika41,anderson58,salaevskii71}) that if $P,Q$ are Gaussians, then the uniformly most powerful test in the fixed-dimension setting under fairly general conditions, is the T-test by \cite{hotelling} : 
$$
T_H := (m_P - m_Q)^T S^{-1} (m_P - m_Q)
$$

where $m_P,m_Q$ and $S$ are the usual empirical estimators of $\mu_P,\mu_Q$ and the joint covariance matrix $\Sigma$. 
In a seminal paper, \cite{bs}, showed that in the high-dimensional setting, the T-test performs quite poorly (specifically when $(n,d)\rightarrow \infty$ with $d/n \rightarrow 1-\epsilon$ for small $\epsilon$). This is intuitively because of the difficulty of estimating the $O(d^2)$ parameters of $\Sigma^{-1}$ with very few samples. Indeed, $S^{-1}$ is not even defined when $d>n$ and is poorly conditioned when $d$ is of similar order as $n$. To avoid this problem, they  proposed to use the test statistic
$$
T_{BS} := (m_P - m_Q)^2 - \mathrm{tr}(S)/n
$$
$T_{BS}$ has non-trivial power when $d/n \rightarrow c \in (0,\infty)$. 
\cite{sd} proposed to instead use $\mathrm{diag}(S)$ instead of $S$ in $T_H$, 
and showed its advantages in certain settings over $T_{BS}$. 
More recently, \cite{cq}, henceforth called CQ, proposed a slight variant of $T_{BS}$, which is a U-statistic of the form
$$
T_{CQ} := \frac1{n(n-1)}\sum_{i\neq j}^n (x_i^T x_j +  y_i^T y_j) -  \frac{2}{n^2}\sum_{i,j=1}^n x_i^T y_j
$$
that achieves the same power without explicit restrictions on $d,n$, but rather in terms of conditions stated in terms of $n,\mathrm{tr}(\Sigma),\mu_P - \mu_Q$.
The settings of under which these various statistics are consistent, or achieve non-trivial power, are slightly complicated to describe, and the reader is referred to their papers for details.

\subsection{Related Work (GA)}
There are many nonparametric test statistics for two-sample testing. One of the most popular tests is the kernel Maximum Mean Discrepancy, henceforth called MMD, proposed in \cite{mmd}. While the technical details of the kernel literature are unnecessary for the purposes of this paper, it suffices to say that the population statistic is
$$
\MMD := \max_{\|f\|_H \leq 1} \E_P f(x) - \E_Q f(y)
$$
where $H$ is a Reproducing Kernel Hilbert Space and $\|f\|_H \leq 1$ is its unit norm ball. There are two related sample statistics, both of which can be shown to be unbiased estimators of $\MMD$. The first is a U-statistic
\begin{eqnarray*}
\MMD^2_u &=& \frac1{n(n-1)}\sum_{i\neq j}^n k(x_i, x_j) \\
&+& \frac1{n(n-1)}\sum_{i\neq j}^n k(y_i, y_j)
 -  \frac{2}{n^2}\sum_{i,j=1}^n k(x_i, y_j) 
\end{eqnarray*}

The second is a linear-time statistic
\begin{eqnarray*}
\MMD^2_l &=& \frac1{n/2} \sum_{i=1}^{n/2} [k(x_{2i-1},x_{2i}) + k(y_{2i-1},y_{2i}) \\
&& - k(x_{2i-1},y_{2i}) - k(y_{2i-1},x_{2i}) ]
\end{eqnarray*}

Note that $T_{CQ}$ is just $\MMD^2_u$ under the linear kernel $k(x,y)=x^Ty$. It is known that in the fixed $d$ setting, the power of both $\MMD_l^2$ and $\MMD_u^2$ approaches $1$ at the rate of $\Phi(\sqrt n)$ where $\Phi$ is the standard normal cdf, see \cite{mmd}. However, nothing is formally known when $d$ could be increasing with $n$.

A recent related manuscript by \cite{powermmd} conducts detailed experiments that demonstrate that in the fixed $n$, increasing $d$ setting, the power of MMD and distance correlation decay \textit{polynomially} in high dimensions against fair alternatives. While the authors  provide some initial insights into this phenomenon for specific examples, there is still no  theoretical analysis of the power of MMD (or any statistic designed for GA) against MSA or GA or any other set of alternatives, in the high dimensional setting.

Another statistic called Energy Distance by \cite{energydistance} is closely tied to the MMD - indeed it has the same form as the MMD with the Euclidean distance instead of a kernel; \cite{lyons} showed that one can also use other metrics instead of the Euclidean distance and \cite{hsiceqdcov} showed that there is a close tie between metrics and kernels for these problems. There has been an initial attempt to characterize some properties of distance correlation (which is a related statistic for the related problem of independence testing) in high dimensions in \cite{dcorhigh}, but no analysis of power is available or easily derivable. There also exist many other tests under GA like the cross-match test by \cite{rosenbaum}, but none of them have been analyzed under HD.

%


\section{Power of $\MMD_l$ (fixed dimension)}

Let us first review the basic argument from \cite{mmd} showing the power in the fixed dimensional setting. It will then become clear what the main difficulties are in establishing results in the high-dimensional setting.

The main tool needed is a simple convergence result of the sample statistic to the population quantity. It becomes convenient to introduce the notation $z_i=(x_i,y_i)$ and $h_{ij} = h(z_i,z_j)$ where
\begin{equation}\label{eq:h}
h_{ij} := k(x_i,x_j) + k(y_i, y_j) - k(x_i,y_j) - k(x_j,y_i). \quad
\end{equation}

Then we can rewrite our test statistic as
\begin{equation}\label{eq:MMDl}
\MMD^2_l = \frac1{n/2} \sum_{i=1}^{n/2} h(z_{2i-1},z_{2i}). 
\end{equation}
Its expectation is $E_{z,z'}h(z,z') = \MMD^2$ and then Corollary 16 of \cite{mmd} states that under both $H_0$ and $H_1$, we have
\begin{equation} 
F := \frac{\sqrt n (\MMD^2_l - \MMD^2)}{\sqrt V} \leadsto N(0,1)
\end{equation}
where $V = 2\mathrm{Var}_{z,z'} h(z,z')$ and $\leadsto$ means convergence in distribution as $n \rightarrow \infty$. Note that $V$ is a constant independent of $n$, and so there exists a constant $z_\alpha$ such that $P( Z > z_\alpha ) \leq \alpha$ when $Z \sim N(0,1)$. Then, the corresponding test rejects $H_0$ whenever 
\begin{equation}\label{eq:test}
\mbox{Test-}\MMD^2_l \quad : \quad \frac{\sqrt n \MMD^2_l}{\sqrt {v}} > z_\alpha
\end{equation}
where $v$ is twice the empirical variance of $h(z,z')$. If $\Pr$ denotes the probability under $H_1$, the power of this test is given by
\begin{eqnarray}
&& \Pr\left( \frac{\sqrt n \MMD^2_l}{\sqrt v} > z_\alpha \right) \label{eq:proof1}\\
&=& \Pr \left(F > \sqrt{\frac{v}{V}}z_\alpha - \frac{\sqrt n \MMD^2}{ \sqrt V} \right)\\
&\xrightarrow{n \rightarrow \infty}& \Pr\left( Z > z_\alpha - \frac{\sqrt n \MMD^2}{\sqrt V}  \right) \label{eq:proof2}\\
&=& 1 - \Phi \left( z_\alpha - \frac{\sqrt n \MMD^2}{\sqrt V}  \right)\\
&=& \Phi \left( \frac{\sqrt n \MMD^2}{\sqrt V} - z_\alpha \right) \label{eq:proof3}
\end{eqnarray}
where $\Phi$ is the standard normal cdf. This behaves like $\Phi(\sqrt n)$ since the population $\MMD^2$ and $V$ are constants that are both independent of $n$.

\subsection{The challenges in high dimensions}

There are several significant difficulties in lifting this argument to the high-dimensional setting. 

\begin{enumerate}
\item[C1.] The population MMD depends on dimension (via the signal strength and bandwidth, as we later show), and one needs to explicitly account for this.
\item[C2.] The variance V also depends on dimension (and the signal strength, and the bandwidth, as we later show), and again one needs to explicitly track this, especially its dependence on dimension.
\item[C3.] In the increasing $d,n$ setting, the limiting distribution is no longer trivially normal, and one needs to establish conditions under which it is indeed normal - the most important question being if the rate of convergence to normality depends on $d$.
\item[C4.] In the increasing $d,n$ setting, one needs to characterize the rate at which $v/V$ still tends to $1$, so that $\sqrt{\frac{v}{V}}z_\alpha$ converges to $z_\alpha$ - since $v,V$ depend on $d$, the key question is again whether the rate of convergence depends on $d$ or not.
\end{enumerate}

We will have to account for each of these challenges explicitly, as we shall see in later sections. Let us first summarize and discuss our assumptions and contributions before we delve into the technical details.

\section{Assumptions and Contributions}

We are now in a position to clearly state our contributions. We focus on analyzing the power of $\MMD_l$ in the high-dimensional setting when $(n,d) \rightarrow \infty$ for the Gaussian kernel with bandwidth $\gamma$, i.e. $k(x,y) = \exp \left(- \frac{\|x-y\|^2}{\gamma^2} \right)$, in the mean-shift setting when $P$ and $Q$ differ in their means. Let us first outline our assumptions below; note that we comment about these assumptions in the next subsection. 

\begin{enumerate}
\item[A1.] $x_i =  U  s_i + \mu_P$ and $y_i =  U t_i + \mu_Q$, where, $s_i,t_i$ are i.i.d random vectors for $i \in \{1,...,n\}$, each having $d$ i.i.d. zero-mean coordinates.
and $U$ corresponds to a $d\times d$ orthogonal rotation i.e. $UU^T = I$.
\item[A2.] The $k$-th central moments of each (i.i.d.) coordinate of $s,t$ exist for $2 \leq k \leq 6$.
\end{enumerate}
Note that the coordinates of $x,y$ need not be independent and $\E_{x \sim P}[X] = \mu_P, \E_{y \sim Q}[Y]=\mu_Q$.  
Denote $\delta := \mu_P - \mu_Q$. Denote the second, third and fourth central moments of each i.i.d. coordinate of $s,t$ by $\sigma^2, \mu_3, \mu_4$. Remember that $\E  h(z_i,z_j)=\MMD^2$ (see Eq.\eqref{eq:MMDl}). Denote the second, third and fourth central moments of $h(z_i,z_j)$ by $V,\tau_3,\tau_4$. Let $\|.\|$ represent the Euclidean norm. Our main contribution is:\\

\begin{theorem}
For the Gaussian kernel with bandwidth chosen as $\gamma = \Omega(\sqrt d)$, under assumptions A1, A2, with $(n,d) \rightarrow \infty$ at any rate, the Test-$\MMD^2_l$ (Eq. \ref{eq:test}) has asymptotic type-1 error 
$
\alpha 
$
and asymptotic power 
$$
\beta \quad =\quad \Phi\left(\frac{\sqrt{n}~ \|\delta\|^2}{ \sqrt{8d\sigma^4   + 8\sigma^2 \|\delta\|^2  }} - z_\alpha \right) 
$$
where $\Phi$ is the cdf of a standard Normal distribution and $z_\alpha$ is the $(1-\alpha)$ quantile of the standard Normal distribution. For finite samples, type-1 error behaves like  $\alpha + 20/\sqrt n$ and the power like $\beta - 20/\sqrt n$. \\
\end{theorem}

The first remarkable point about this theorem is that the power is independent of bandwidth $\gamma$, as long as $\gamma = \Omega(\sqrt d)$. Such behavior has already been noted (but not explained) in the experiments of \cite{powermmd} and we will verify this carefully in our experiments section. While this may not hold true for other kernels, like the Laplace kernel $k(x,y) = \exp \left(- \frac{\|x-y\|_1}{\gamma} \right)$, or against more general alternatives, it is both surprising and interesting that this is the case for the Gaussian kernel under MSA. As discussed later, this theorem applies to the bandwidth chosen by the so-called \textit{median heuristic}; see \cite{learningkernels}. It implies that the median heuristic provides an arguably safe choice in the light of having no further information, and also why it works reasonably well in practice/simulations.

If we consider the \textit{signal to noise ratio} (henceforth called SNR) to be defined as $\Psi := \|\delta\|/\sigma$, then  focusing on the more important first term, the power behaves like
$$
\Phi \left(\frac{\sqrt{n}~ \Psi^2 }{ \sqrt{8d   + 8\Psi^2  }} -z_\alpha \right).
$$

From this, we get the following two corollaries. The first applies to the small SNR regime (which includes the \textit{fair} alternative setting, see \cite{powermmd} for details), and the second applies when SNR is large.

\begin{corollary}
When the signal to noise ratio $\Psi$ is small, specifically $\Psi = o(d^{1/2})$, the power goes to 1 at the rate of $\Phi( \sqrt{n}\Psi^2/\sqrt{d}) $. \\
\end{corollary}

\begin{corollary}
When the signal to noise ratio $\Psi$ is large, specifically $\Psi = \omega(d^{1/2})$, then the power goes to 1 at the rate of $\Phi(\sqrt n \Psi)$, independent of $d$.
\end{corollary}

Note that the switch in behavior between the two corollaries occurs at $\Psi$ being on the order of $d^{1/2}$, and at this point the prediction of the two corollaries match - hence one could use $O, \Omega$ instead of $o, \omega$ for describing growth of $\Psi$ in both corollaries.


\subsection{Remarks about assumptions}

Assumptions (A1,A2) are general enough for the predictions made by our theorem to be accurate and representative of observed behavior. We will verify the predictions of the theorem, corollaries (and later lemmas) in our simulations.
\begin{enumerate}
\item[A1.] While the coordinates of $x,y$ need not be independent, the first assumption does restrict their covariances to be $\sigma^2 I$. We note that \cite{dcorhigh} makes a more restrictive assumption of independent coordinates, while Assumption (a) in \cite{bs} and Eq.(3.1) in \cite{cq} assume the same model as we do but don't require spherical covariance. However, our assumption is truly only for mathematical convenience; if we instead had $UD^{1/2}$ in A1, where $D$ is a diagonal rescaling, all our calculations can still be carried out, but would be more tedious since the coordinates of $D^{1/2}s$ are still independent but \textit{not} identically distributed, and we would need to track $\sigma^2_j, \mu_{3j},\mu_{4j}$ in Appendix Sections 3-6.
\item[A2.] 
 The existence of third and fourth moments is needed for calculating population MMD and variance terms, as well as for the Berry-Esseen lemma to control the deviation from normality, and the convergence of $v$ to $V$. The existence of the sixth moment is needed to bound the Taylor expansion residual term in all our calculations. Note that CQ needs the existence of eighth moments, and BS assume the existence of fourth moments (see Eq. (3.2) in \cite{cq}) and Assumption (a) in \cite{bs}.


\end{enumerate}


\subsection{Remark about bandwidth choice}

Remember that the power is independent of the bandwidth $\gamma$, as long as $\gamma = \Omega (\sqrt d)$. This restriction of $\gamma = \Omega(\sqrt d)$ is to allow us to control the residual term in the Taylor expansion of the Gaussian kernel. However, it is not very restrictive, since smaller $\gamma$ typically leads to  worse power. Specifically, we note that the experiments in \cite{powermmd} for mean-shift alternatives show  convincingly that when $\gamma$ is chosen to be a constant or $d^\alpha$ for $\alpha < 0.5$ (including constant $\gamma$), then the power of MMD is  poor, while when the highest power occurs for values $\alpha \geq 0.5$. Hence our choice covers most reasonable choices of bandwidth. Furthermore, one of the most popular methods for bandwidth selection is called the \textit{median heuristic}, see \cite{learningkernels}, where one chooses the bandwidth as the median of  distances between all pairs of points. A simple calculation shows $\E_{x \sim P, y \sim Q}\|x - y\|^2 = 2\sigma^2d + \|\mu_P - \mu_Q\|^2$, so generally speaking the median heuristic chooses $\gamma$ of the same order as $\sigma \sqrt {2d}$ (or larger if $\|\mu_P - \mu_Q\|$ is large). 



\subsection{Comparisons to CQ}

The assumptions in CQ, BS, SD are slightly differently stated from our results here. However, their results can broadly be compared to ours. We can summarize the most recent results, those of CQ, under (A1) and (A2) in the following two observations.  

%
%

The first observation follows from Eq. (3.11) in \cite{cq} which applies to the small SNR regime dictated by Eq. (3.4).

\begin{observation}
When the signal to noise ratio $\Psi$ is small, specifically $\Psi = o(\sqrt{d/n})$, the power goes to 1 at the rate of $ \Phi(n \Psi^2/\sqrt{d}) $. 
\end{observation}

We believe there is a mistake in the derivation of Eq. (3.12) in \cite{cq} which applies in the small SNR regime dictated by Eq. (3.5). We describe this in more detail in the Appendix Section 1, and just summarize the corrected resulting observation below.

\begin{observation}
When the signal to noise ratio $\Psi$ is large, specifically $\Psi = \omega(\sqrt{d/n})$, then the power goes to 1 at the rate of $\Phi (\sqrt n \Psi)$, independent of $d$.
\end{observation}

Comparing these expressions with Corollary 1 and 2, it is clear that CQ has an advantage over $\MMD_l$ in the low-SNR setting. For example, when $n=d$ and the SNR  $\Psi$ is constant, the power of CQ can increase $\sqrt n$ times faster than that of $\MMD_l$ but when the SNR is $\omega(d^{1/2})$, the power of both methods scales in the same fashion. This advantage for low SNR might be wiped out by considering $\MMD^2_u$ - ascertaining if this is the case is an important direction of future work. The main technical challenge is understanding the limiting distributions of general degenerate U-statistics in high dimensions (which in fixed dimensional setting is an infinite sum of $\chi^2$s; see \cite{serfling}, Section 5.5.2).

We now provide the proof of Theorem 1 and then verify all our claims in simulations, to convincingly show that these expressions are tight up to constant factors.


\section{Proof of Theorem 1}

We split the proof into four subsections, one for each of the challenges (C1)-(C4). For C1 and C2, we need to calculate the first two moments of $h$, introduced in Eq.\eqref{eq:h}, for which the main tool we use is Taylor expansions (whose validity is explained in Appendix Section 2), following which the results follow after a sequence of tedious calculations and detailed book-keeping. For C3 and C4, we need to bound the third and fourth moments of $h$. The main tool used for C3 is a Berry-Esseen theorem which helps us track the deviation from normality at finite samples, and C4 is tackled by Chebyshev's inequality once we have a handle on the variance of $v$.  Most of the details will be deferred to the Appendix, but we will outline the main steps of the derivations here.

\subsection{The Population $\MMD$}

The main takeaway point of the following lemma is the dependence of population $\MMD^2$ on the bandwidth $\gamma$ and the signal strength $\|\delta\|$ (recall $\delta := \mu_P - \mu_Q$). If $p,q$ are the pdfs of $P,Q$, then note that the population $\MMD^2$ with the Gaussian kernel is given by 
\begin{eqnarray*}
\int_{\R^d} e^{-\frac{\|x-y\|^2}{\gamma^2}} (p(x) p(y) + q(x) q(y) - 2 p(x)q(y)) dx dy 
\end{eqnarray*}


\begin{lemma}
Under (A1),(A2), and when $\gamma = \Omega(\sqrt d)$ we have
$$
\MMD^2 = \frac{2\|\delta\|^2}{\gamma^2} (1 + o(1)).
$$
\end{lemma}

\begin{proof}
We defer details to the Appendix Section 3. 
On using Taylor's expansion for the Gaussian kernel, the terms in the aforementioned $\MMD^2$ expression can be approximated by bounding higher order residual terms.  
We prove that the first $\MMD^2$ term is
\begin{eqnarray*}
\int_{\R^d} e^{-\frac{\|x-y\|^2}{\gamma^2}} p(x) p(y) dx dy =  \left( 1 - \frac{2\sigma^2}{\gamma^2}\right)^d.\\
\end{eqnarray*}
Using similar techniques we can also deduce:
\begin{eqnarray*}
\int_{\R^d} e^{-\frac{\|x-y\|^2}{\gamma^2}} p(x) q(y) dx dy = \prod_i \left( 1 - \frac{2\sigma^2}{\gamma^2} - \frac{\delta_i^2}{\gamma^2} \right).
\end{eqnarray*}
Combining these, again using Taylor expansions, gives us our expression.
\end{proof}

\subsection{The Variance}
As argued earlier, the variance is given by $2V/n$ where $V = \mathrm{Var}_{z,z'}h(z,z')$. The takeaway points of the following lemma are the identical dependence that $\sqrt V$ has  on bandwidth $\gamma$ as the $\MMD^2$ (which then causes their ratio to be essentially independent of $\gamma$), and also the role played by dimension and the signal strength in determining the variance. 

\begin{lemma}
Under (A1),(A2), and when $\gamma=\Omega(\sqrt d)$, we have
$$
V = \frac{16d\sigma^4 + 16 \sigma^2 \|\delta\|^2 }{\gamma^4} (1 + o(1)).
$$
\end{lemma}

\begin{proof}
Note that $V = \E_{z,z'} h^2(z,z') - \MMD^4$ since $\MMD^2 = \E_{z,z'} h(z,z')$. Let us focus on the first term:

\begin{align*}
 \E_{z,z'}[h^2(z,z')] &= \textcolor{black}{\E_{x,x'\sim P} k^2(x,x')} + \textcolor{black}{\E_{y,y' \sim Q} k^2(y,y')}\\
  &+ \textcolor{black}{2\E_{x\sim P,y \sim Q} k^2(x,y)} \\
 &+ \textcolor{black}{2 \E_{x,x' \sim P,y,y' \sim Q} k(x,x')k(y,y')} \\
 &+ \textcolor{black}{2 \E_{x,x' \sim P,y,y' \sim Q} k(x,y')k(x',y) }\\
& - \textcolor{black}{4\E_{x,x' \sim P,y \sim Q} k(x,x')k(x,y)} \\
&- \textcolor{black}{4\E_{x \sim P,y,y' \sim Q} k(x,y)k(y,y') }
\end{align*} 
Hence, there are five different kinds of terms to calculate (the first and last two are similar). Combining these gives us our solution. The details are tedious and hence are given in the Appendix Section 4. 
\end{proof}


\subsection{The Berry-Esseen Bound}

\begin{lemma} Under (A1), (A2), and when $\gamma = \Omega(\sqrt d)$, we have
\begin{eqnarray*}
\sup_t \left|
\mathbb{P}\left( \frac{\sqrt{n/2}(\MMD^2_l - \MMD^2)}{\sqrt V} \leq t\right) -\Phi(t) \right| \leq \frac{20}{\sqrt{n}}
\end{eqnarray*}
\end{lemma}

\begin{proof}
The Berry-Esseen Lemma (see for example Theorem 3.6 or 3.7 in \cite{cgs}), when translated to our problem, essentially yields the above lemma, except that the right hand side is
\begin{equation}\label{eq:BE}
10 \frac{\xi_{3}}{V^{3/2} \sqrt{n}}
\end{equation}
where $\xi_3=\E[|h(z,z') - \E h(z,z')|^3]$, and the constant 10 is not optimal. Note that $\xi_3 \neq \tau_3$ (third central moment of $h$) due to the absolute value sign. Given that we have the mean and second central moment of $h$ ($\MMD^2$ and $V$ respectively), one might imagine using similar techniques to calculate $\xi_3$. However, the absolute value poses a problem, and so we must take an alternate route. Specifically, tedious calculations in the Appendix Section 5 prove that $\tau_4$ (the fourth central moment of $h$) is bounded as
$$
\tau_4 \leq (4+o(1)) V^2,
$$
allowing us to bound $\xi_3$ as
$$
\xi_3 \leq \sqrt{\tau_4} \sqrt{V} \leq 2 V^{3/2}
$$
since $\E|X|^3 \leq \sqrt{\E|X|^4} \sqrt{\E|X|^2}$ by Cauchy-Schwarz.  
Substituting into Eq.\eqref{eq:BE} gives us our Lemma.

\end{proof}

The main challenge involved is in proving that the ratio $\xi_3/V^{3/2}$ is independent of $d$.
Note that a very crude bound of $|h - \E h| \leq 4$ (since $e^{-z} \leq 1$) gives us $\xi_3 \leq 4V$, which would yield a dimension dependence due to an extra $\sqrt V$ factor, but because $\tau_4$ (and hence $\xi_3$) has exactly the right scaling with $V$, the dependence on $V$ (and hence, importantly, the dimension) cancels out and our Lemma follows. This is only one of the reasons we needed a bound on $\tau_4$, the other appearing in the next lemma.


\subsection{Bounding $\sqrt{v/V}$}

Recall that $v$ is  the empirical estimator of $V$ - it is an empirical average of $n/2$ unidimensional terms. The subtlety is that $v$ depends on $d$ since $V$ depends on $d$. 

\bigskip

What matters is whether the rate of convergence of their ratio to $1$ depends on $d$ - fortunately it does not.

\begin{lemma} Under (A1),(A2), and when $\gamma = \Omega(\sqrt d)$, we have 
$$
\sqrt{v/V} = 1 + O_P(1/n^{1/4})
$$
\end{lemma} 
\begin{proof}
Using $k=2$ in Theorem A of Section 2.2.3 in \cite{serfling}, the bias of $v$ is given by
$$
\E[v] - V = -\frac{2V}{n}
$$
and its variance is given by
$$
\mathrm{var}(v) = \frac{\tau_4 - V^2}{n} \leq \frac{3V^2}{n}
$$
both up to smaller order terms (where the inequality follows from the previous lemma).

Then, it is easy to see that $v = V\left(1 + O_P\left(\frac1{\sqrt  n}\right)\right)$, i.e. $v - V = O_P(V/\sqrt n)$. This is because for any $\epsilon > 0$,
\begin{eqnarray*}
&& P\left(\left|\frac{v-V}{V/\sqrt n} \right| > \frac{3 + 2\sqrt{\epsilon}}{\sqrt \epsilon} \right)\\
 &=& P\left(|v - \E[v]| > \frac{3V + 2V\sqrt{\epsilon}}{\sqrt {n \epsilon}} - \frac{2V}{n} \right)\\
 &\leq& \frac{\mathrm{var}(v)}{\left(\frac{3V}{\sqrt {n \epsilon}} + \frac{2V}{\sqrt n} - \frac{2V}{n} \right)^2} \\
 &\leq& \epsilon
\end{eqnarray*}
where we used Chebyshev's inequality, and the second inequality follows since $\frac{3V}{\sqrt {n \epsilon}} + \frac{2V}{\sqrt n} - \frac{2V}{n} \geq \frac{3V}{\sqrt {n \epsilon}}$.

\end{proof}

At this point we have all the key elements of the proof of Theorem 1. Specifically, equations \eqref{eq:proof1} to \eqref{eq:proof3} follow exactly as written, with the exception of \eqref{eq:proof2} holding even with a $\xrightarrow{n,d \rightarrow \infty}$ - note that this step allows $n,d$ to grow at any relative rate to $\infty$ precisely because the rate at which $Q$ converges to the standard normal $Z$ (Berry-Esseen bound) and the rate at which $v/V$ converges to 1, were both independent of $d$ and only needs $n \rightarrow \infty$. The dependence on $d$ only enters through the $\MMD^2$ and its variance.

This concludes the proof of Theorem 1. One can also write down the finite sample type-1 error rate as being at most $\alpha + 20/\sqrt n$ and the finite sample power as being at least $\beta - 20/\sqrt n$, where the additional error is introduced due to the Berry-Esseen bound (whose constants we don't optimize, but could be tightened to about 15 instead of 20).

We now confirm the tightness of all the predictions in this section by detailed simulations in the next section.

\section{Experiments}

Our aim in this section is to confirm the theoretical predictions made by our lemmas and theorems. The most important claims to address are that the Berry-Esseen bound is independent of $d$, the null and alternate distributions are indeed normal even in the extreme case when $n$ is fixed and $d$ is increasing, the ratio of $\MMD^2/\sqrt V$ is (essentially) independent of the bandwidth, and finally the final power expression is (essentially) independent of the bandwidth and has the exact predicted scaling as given by our expressions.

\subsection{Berry-Esseen bound is independent of $d$}
Since the calculations of $\tau_4$ are rather tedious, let us also verify the prediction made in Subsection 4.3 that $\xi_3/V^{3/2}$ is constant and independent of dimension (remember that the ratio involves population quantities). To verify this, we draw 1000 samples from $P,Q$, and calculate the empirical ratio for $d$ ranging from 40 to 1000, in steps of 20. We make 3 sets of choices for $P,Q$ - standard normals with $\gamma = d^{0.75}$, $t_4$ distribution with $\gamma = d^{0.5}$ and $t_4$ distribution with $\gamma = d$. The reason we use $t_4$ ($t$ distribution with 4 degrees of freedom) is because it does not have a finite fourth moment $\tau_4$. We find that in all 3 cases, the ratio is a constant of about 1.65, showing that our prediction is extremely accurate. Also, while our proof proceeded via bounding $\tau_4$, it seems to hold true even when higher moments than $3$ don't exist, since it holds for the $t_4$ distribution. The spikes are because we calculate a single empirical ratio at each $d$.

\begin{figure} [h!]
\centering
\includegraphics[width=0.4\linewidth]{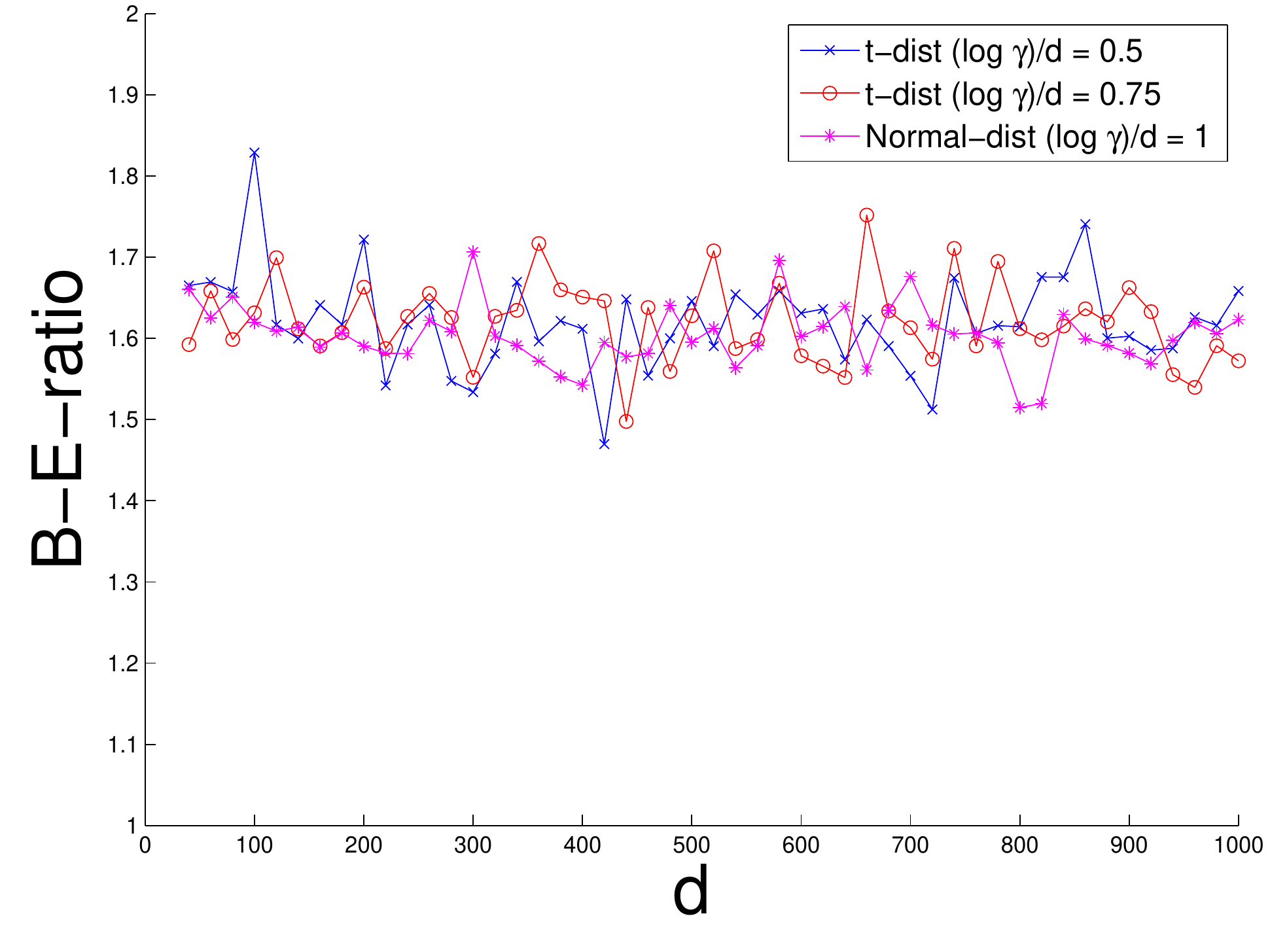}
\caption{The empirical Berry-Esseen ratio $\xi_3/V^{3/2}$ vs dimension, when $n=1000$ for the distributions $t_4, t_4$ and normal, with bandwidths $d^{0.5},d,d^{0.75}$ respectively. }
\label{fig:BEratio}
\end{figure}

\subsection{Normality of null/alternate distributions}

Let us now verify that the null and alternate distributions are indeed (almost) standard normal when $n$ is held constant and $d$ is increased. We do this by fixing $n=50$, and choosing $d\in \{50,100,200\}$  and calculating our test statistic $\sqrt n \MMD_l^2 / \sqrt v$. We experimentally approximate the null and alternate distributions by repeating this process 1000 times; the histogram obtained is compared to a normal by plotting a standard normal quantile-quantile plot. The overlapping straight lines indicate that each of the null and alternate distributions (for three different $d$ values) are almost exactly standard normal even at a small value of $n$ like $50$. This agrees with our derivation that the Berry-Esseen constant is very small and normality is achieved soon.

\begin{figure} [h!]
\centering
\includegraphics[width=0.4\linewidth]{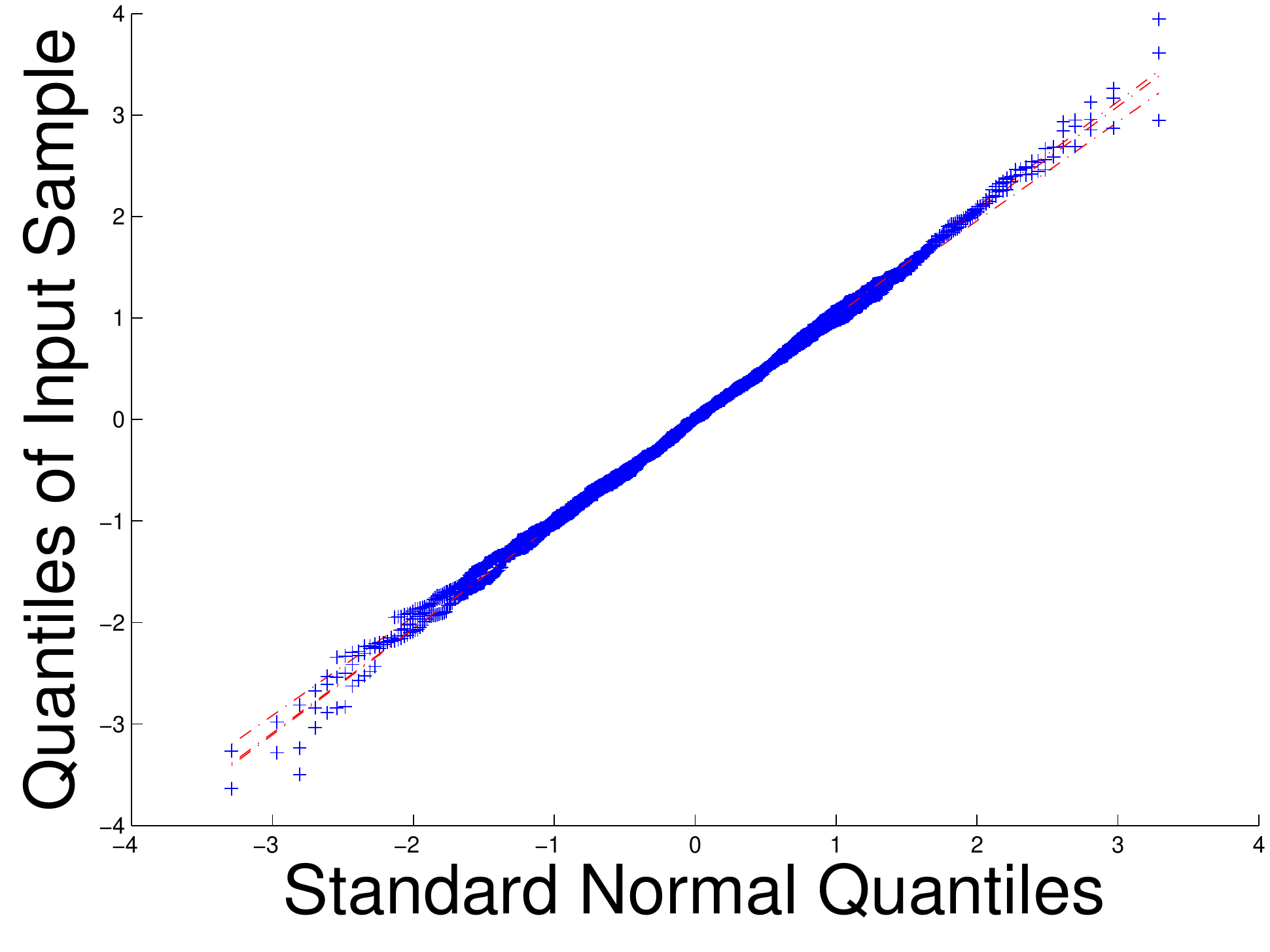}
\includegraphics[width=0.4\linewidth]{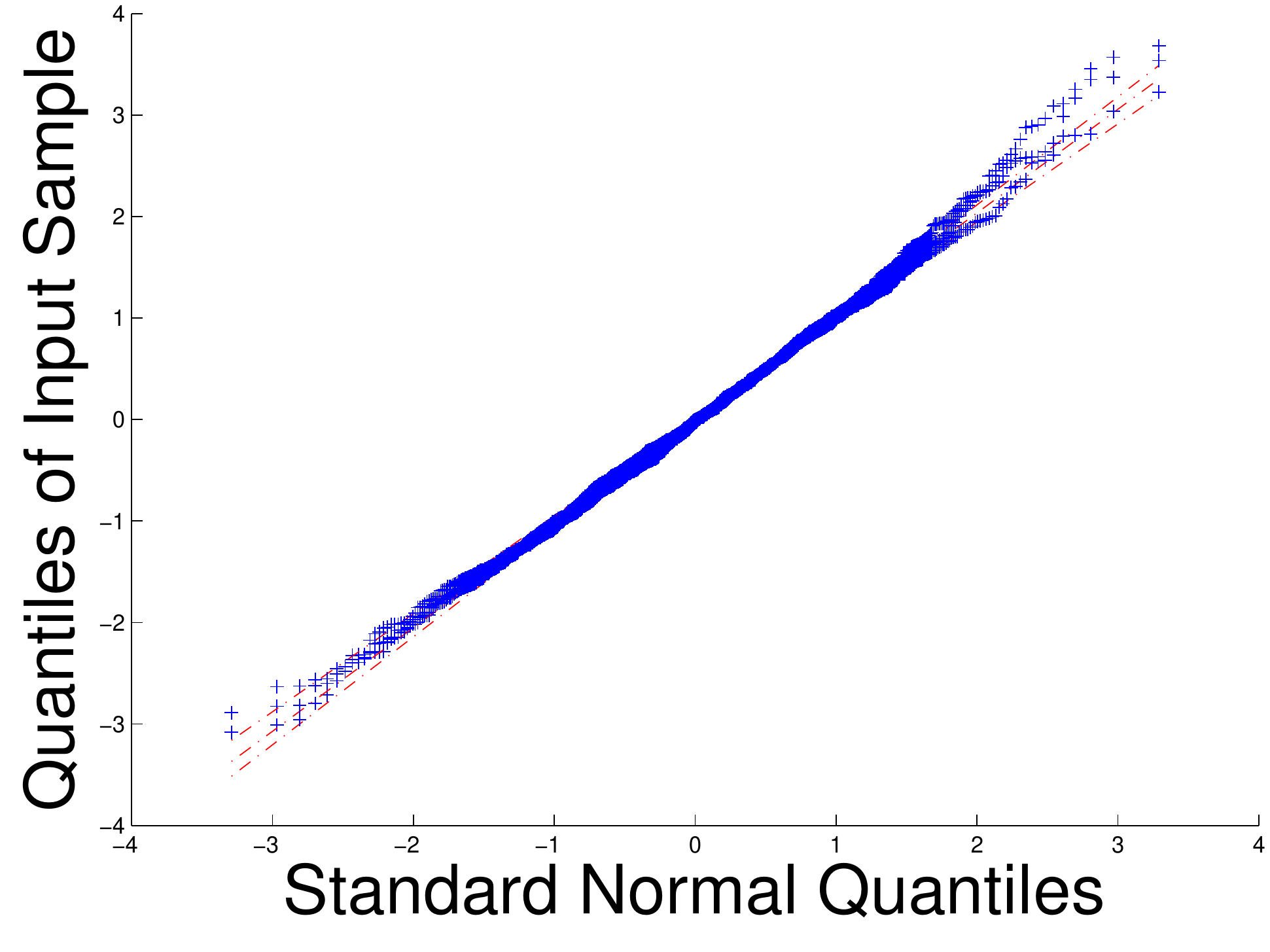}
\caption{A normal quantile-quantile plot of null (left) and alternate (right) distributions of our test statistic for $d=50,100,200$ when $n=100$ (1000 repetitions).}
\label{fig:BEratio}
\end{figure}

\subsection{$\MMD^2/\sqrt V$ is independent of bandwidth}

\begin{figure} [h!]
\centering
\includegraphics[width=0.4\linewidth]{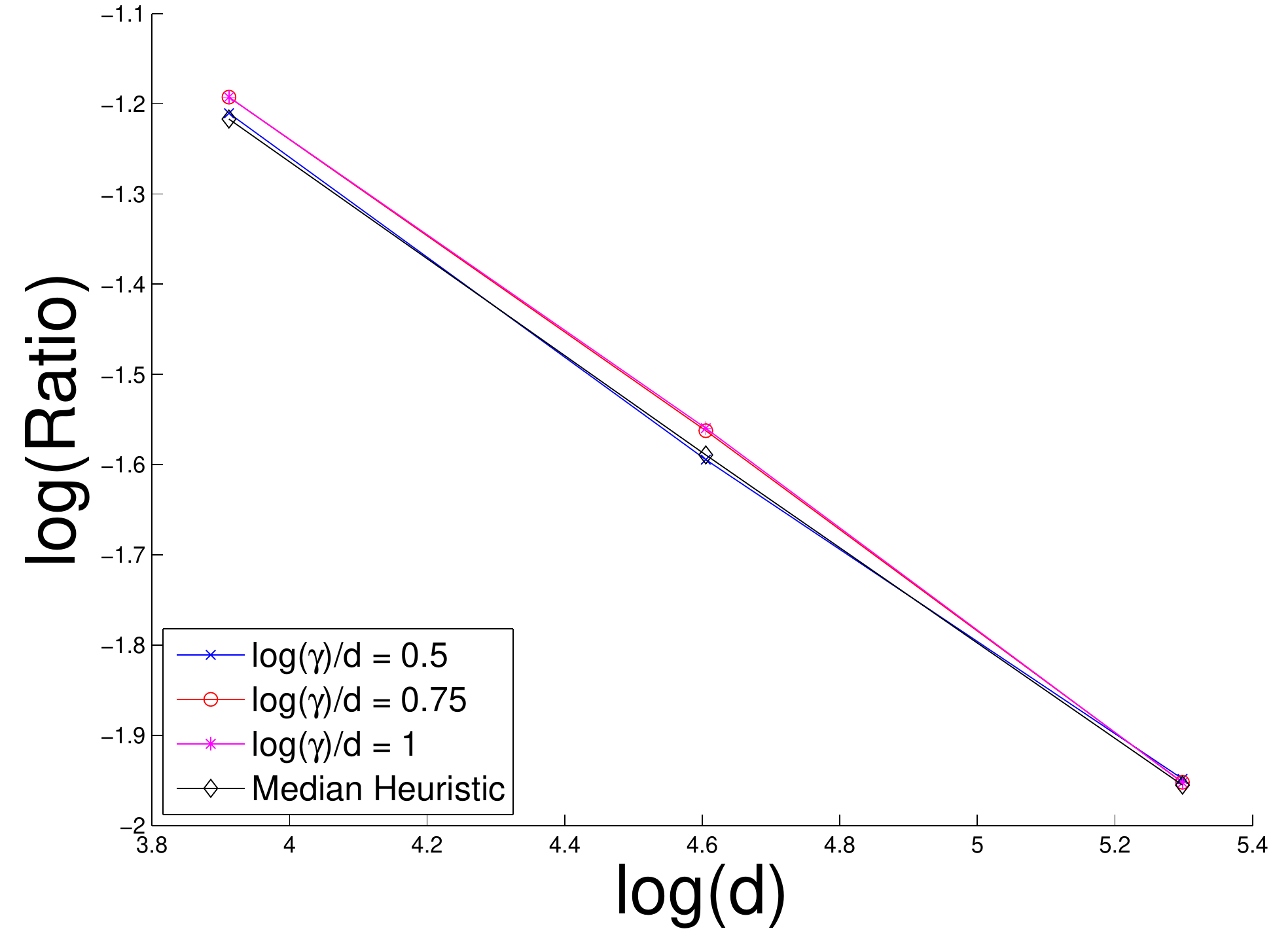}
\caption{A log-log plot of $ \MMD^2 / \sqrt{V}$ vs dimension for different bandwidth choices when $\Psi=1$ and $n$ is large. Note that the slope is $-0.5$, independent of $\gamma$.}
\label{fig:BEratio}
\end{figure}

Our first two lemmas together imply that the ratio $ \MMD^2 / \sqrt{V}$ is independent of $\gamma$ as long as $\gamma = \Omega(\sqrt d)$. To test this, we actually calculate this ratio for $\gamma = d^{0.5}, d^{0.75}, d$. Remember that these are population quantities - we will estimate the ratio using sample quantities using a large $n$, when $\Psi=1$. We plot the obtained log-ratio against log-dimension in Figure 3, showing that the power scales  as $1/\sqrt{d}$ as predicted.

\subsection{The scaling of power with $n,d$}

Here are a few testable predictions of Theorem 1:
\begin{enumerate}
\item When $n=50$ and $\Psi=2.5$, the power should decrease as $1/\sqrt d$ (Corollary 1).
\item When $n=50$, and $\Psi = d^{1/4}$, then the power should be a constant (Corollary 1).
\item When $n=d$, and $\Psi=2$, the power should stay constant (Corollary 1).
\item When $n=d$, and $\Psi = 0.3 d^{1/2}$, then the power should increase as $\sqrt d$ (Corollary 2).
\end{enumerate}
\begin{figure} [h!]
\centering
\includegraphics[width=0.4\linewidth]{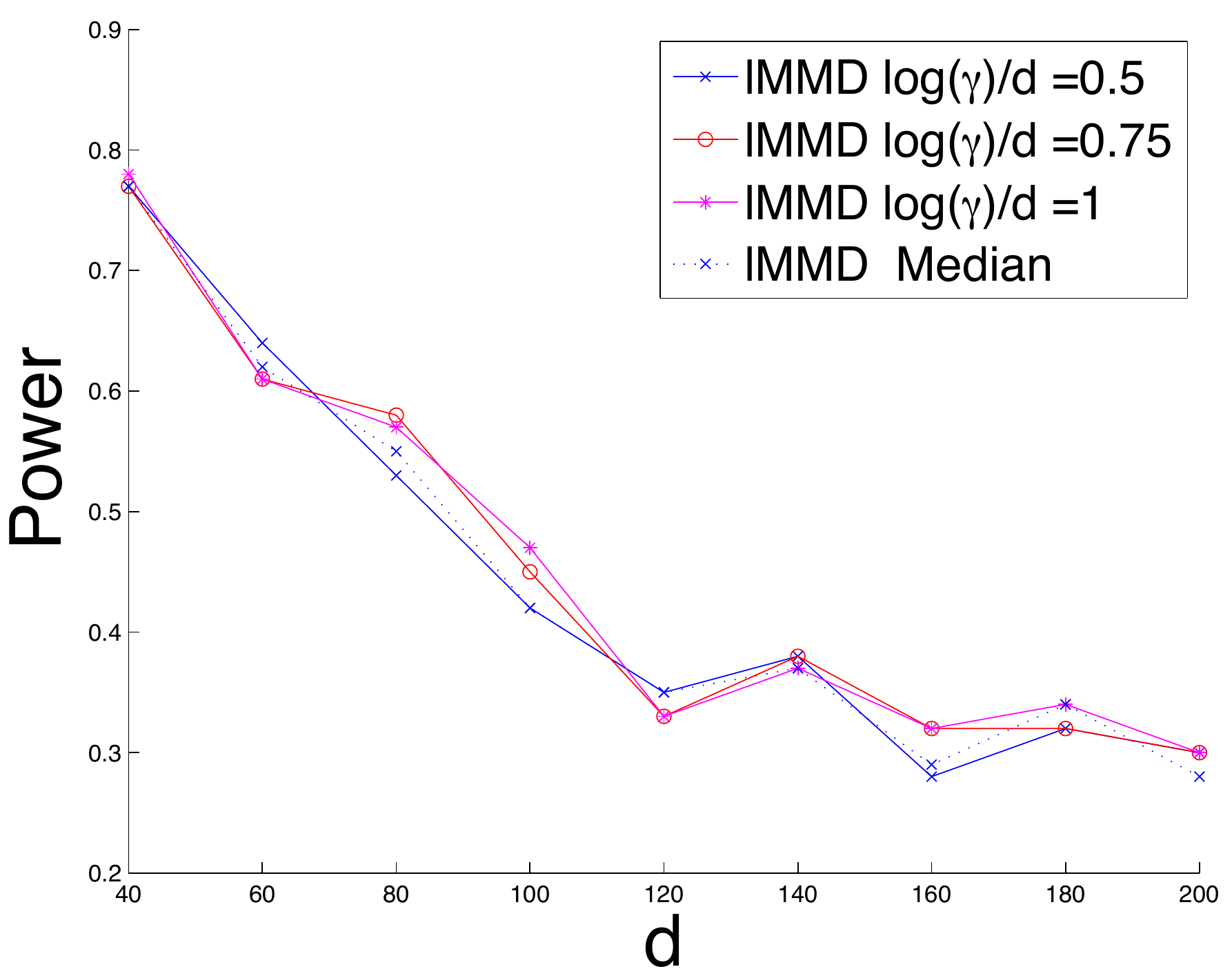}
\includegraphics[width=0.4\linewidth]{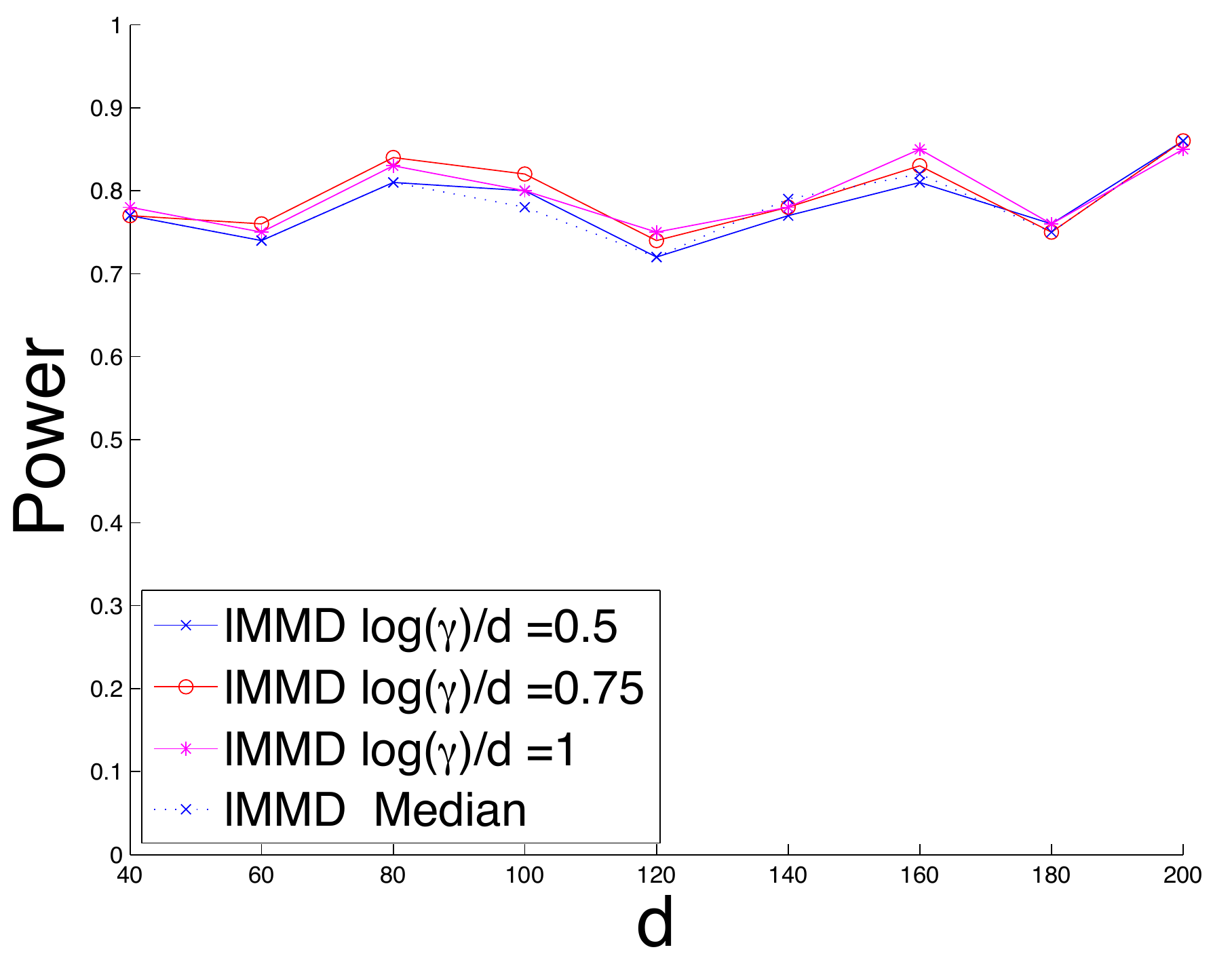}
\includegraphics[width=0.4\linewidth]{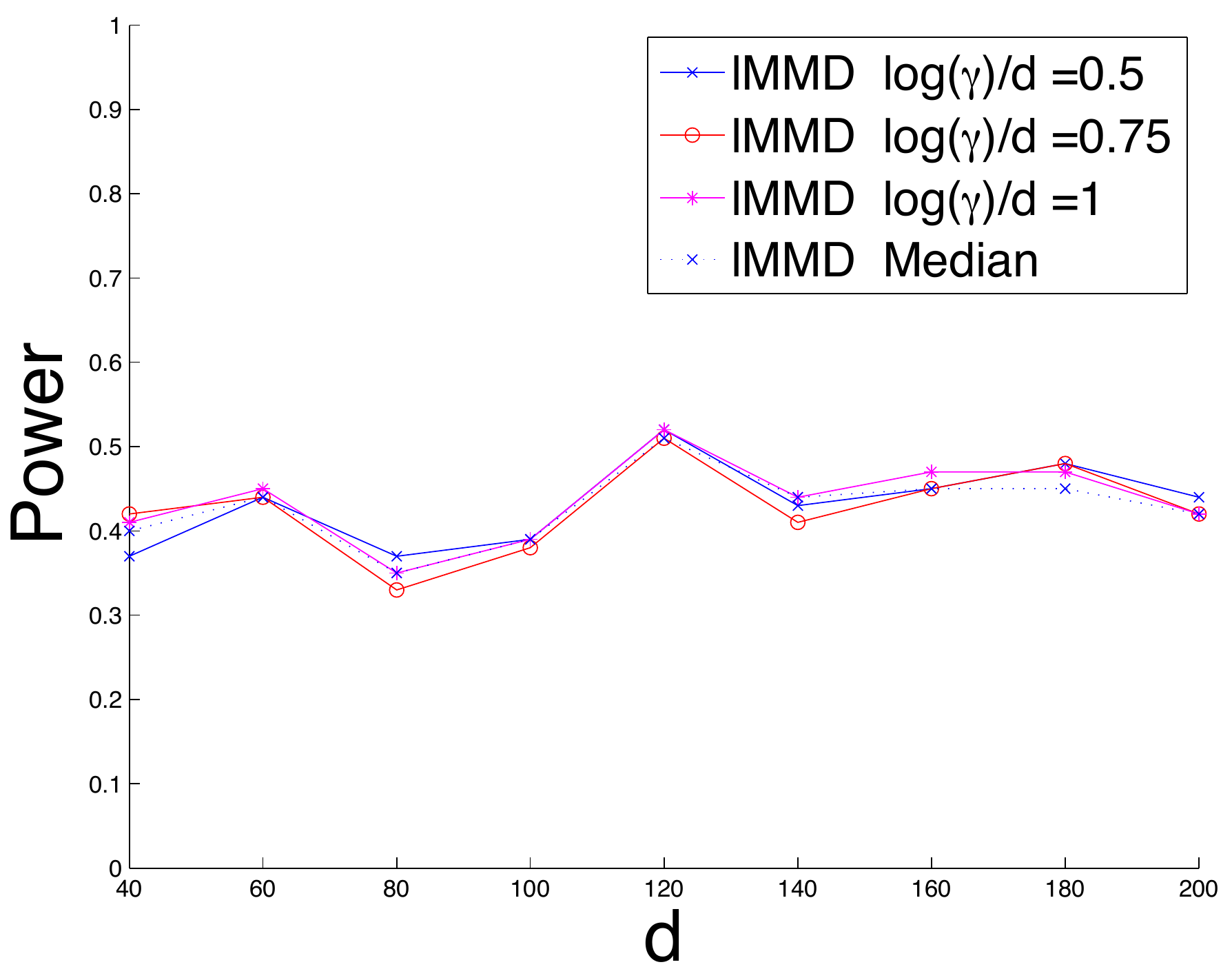}
\includegraphics[width=0.4\linewidth]{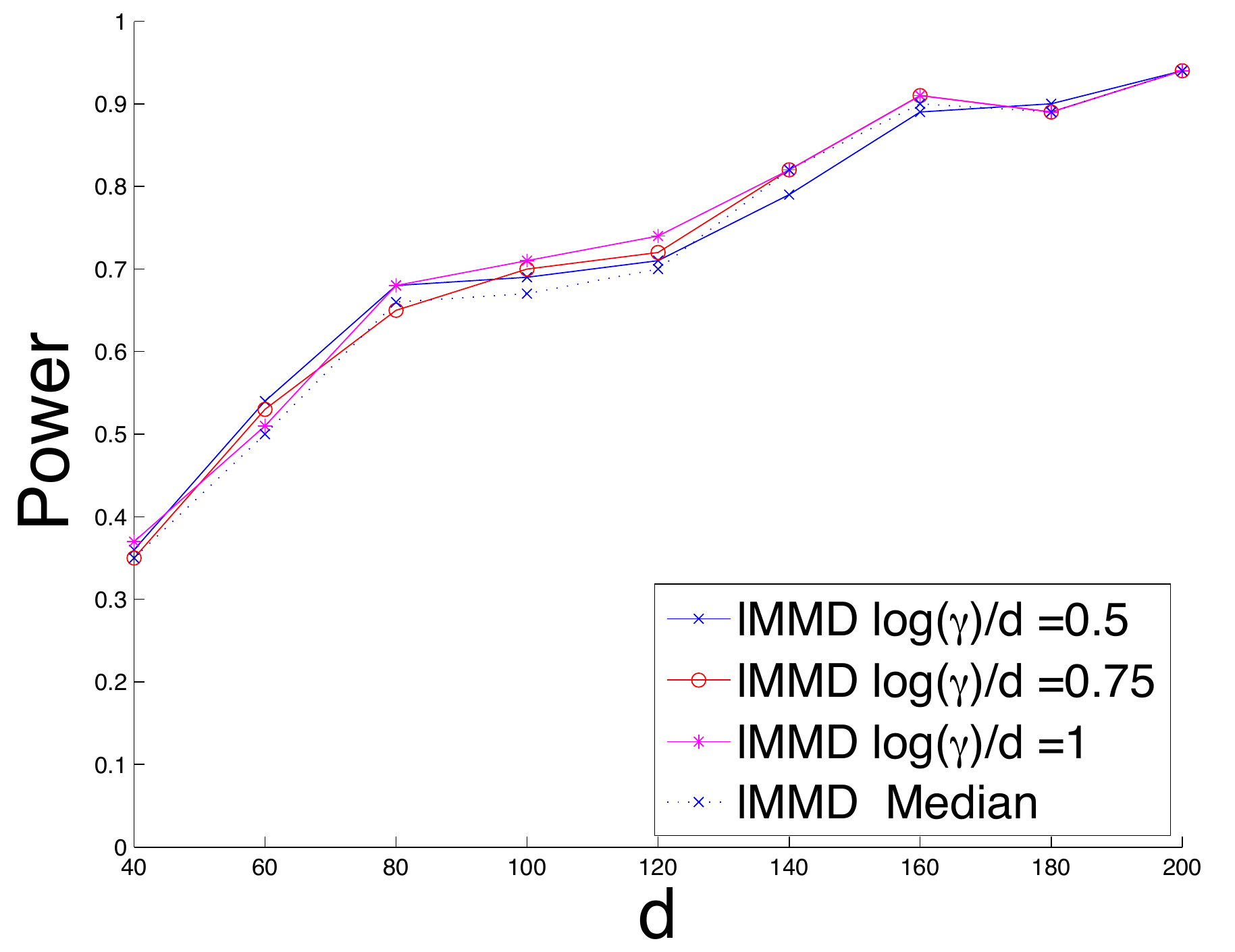}
\caption{All plots show power vs $d$ for different $\gamma \in \{\mbox{median},d^{0.5},d^{0.75},d\}$ for $d=40$ to $200$ in steps of 20. From top left to bottom right are the settings 1-4, with $P,Q$ being Gaussians. The power is estimated over 100 repetitions at each $d$. }
\label{fig:BEratio}
\end{figure}
From Figure 4, we infer that the precise form of Theorem 1 (and Corollaries 1,2) is extremely accurate, even at small $n$ and significantly larger $d$, including that it is independent of the bandwidth $\gamma$ as predicted, as long as $\gamma = \Omega(\sqrt d)$.
\section{Conclusion}
This paper has two main novelties - the first is to  precisely characterize how a nonparametric two sample test, which is consistent in fixed dimensions against general alternatives, performs against a mean-shift alternative; the second is to perform the analysis in the significantly more difficult high-dimensional regime. 

Future work involves understanding $\MMD_u$, but the limiting distributions of general U-statistics are be difficult to ascertain in high dimensions. Another direction involves the study of \textit{sparse} alternatives, where $\delta$ is sparse, as done by \cite{caietal14}. Lastly, minimax lower bounds are required to understand the tradeoffs involved between $\Psi,d,n$.

\bibliographystyle{natbib}
\bibliography{mmd}

\appendix

\input{AISTATS-power-App-arxiv.tex}

\end{document}

%% file: AISTATS-power-App-arxiv.tex
%
%
%
%
%
%
%

\section{The Power of CQ for high SNR}
Let us first briefly describe what we believe is an important mistake in \cite{cq} - all notations, equation numbers and theorems in this paragraph refer to those in \cite{cq}. Using the test statistic $T_n/\hat \sigma_{n1}$ defined below Theorem 2, we can derive the power under assumption (3.5) as

\begin{eqnarray*}
&&P_1 \left( \frac{T_n}{\hat \sigma_{n1}} > \xi_\alpha \right) =\\ 
&=&P_1 \left( \frac{T_n - \|\mu_1 - \mu_2\|^2}{\hat \sigma_{n2}} > \frac{\hat \sigma_{n1}}{\hat \sigma_{n2}} \xi_\alpha -\frac{ \|\mu_1 - \mu_2\|^2}{\hat \sigma_{n2}} \right)\\
& \rightarrow& \Phi \left( \frac{ \|\mu_1 - \mu_2\|^2}{\hat \sigma_{n2}} \right) \mbox{ (the denominator is \textit{not} $\hat \sigma_{n1}$)}\\
 &=&\Phi \left( \frac{\sqrt{n} \|\mu_1-\mu_2\|^2}{\sqrt{(\mu_1 - \mu_2)^T \Sigma (\mu_1 - \mu_2)}} \right) 
 \end{eqnarray*}
which should be the expression for power that they derive in Eq.(3.12), the most important differnce being the presence of $\sqrt n$ instead of $n$ in the numerator. They also do not have an explicit Berry-Esseen bound dealing with the deviation from normality. \\

\section{Remarks for this Appendix}
\subsection{Taylor Expansion}

In all our calculations, we use the Taylor expansion for the function $e^{-x}$ around 0. More specifically, we have

\begin{align*}
\int_u \int_v e^{-\frac{(u - v)^2}{\gamma^2}} p_i(u) q_i(v) du dv = \int_u \int_v \left[1 - \frac{(u - v)^2}{\gamma^2} + \frac{e^{-\frac{\lambda_{uv}(u - v)^2}{\gamma^2}} (u-v)^4}{2\gamma^4} \right] p_i(u) p_i(v) du dv
\end{align*}
where $\lambda_{uv} \in [0,1]$. The above equality follows from the exact formula for Taylor expansions having exact residuals. Note that
$$
e^{-\frac{\lambda_{uv}(u - v)^2}{\gamma^2}} \leq 1.
$$ 
When $\gamma = \Omega(\sqrt{d})$ and fourth moments of the distributions $p_i$ and $q_i$ exist, the above integral becomes
$$
\int_u \int_v e^{-\frac{(u - v)^2}{\gamma^2}} p_i(u) q_i(v) du dv = \int_u \int_v \left[1 - \frac{(u - v)^2}{\gamma^2} \right] p_i(u) p_i(v) du dv + o\left(\frac{1}{\gamma^2}\right)
$$

Similarly, an higher order expansion can also be obtained by assuming existence of sixth order moments. For ease of exposition, we drop $o(1)$ throughout our calculations. To emphasize this issue, we use $\approx$ symbol in our calculations to indicate that the $o(1)$ term is ignored. 

\subsection{Independent Coordinates}

In our calculations, we assume that the coordinates of $x,y$ are independent and that their central moments are $\sigma^2,\mu_3,\mu_4$. In other words, we use $U=I$ in Assumption 1 to derive expressions in this Appendix. However, this is only for ease of exposition and all our proofs hold even when $U \neq I$. This can be seen from the following argument.\begin{align*}
\|x - y\|^2 &= \|Us + \mu_P - Ut - \mu_Q\|^2 \\
&= \left\|U\left(s - t + U^\top (\mu_P - \mu_Q)\right)\right\|^2 \\
&= \|(s + U^\top \mu_P) - (t + U^\top \mu_Q)\|^2 \\
&= \|s' - t'\|^2.
\end{align*}
where $s' = s +  U^\top \mu_P$ and $t' = t + U^\top \mu_Q$.
Since $U^T\mu_P$ and $U^T \mu_Q$ are just rotated mean vectors, the coordinates of $s'$ and $t'$ are independent (since the coordinates of $s,t$ are independent in assumption A1) and $s',t'$ still have the same central moments as $s,t$. 

Using the above relation, we can rewrite our calculations involving $e^{-\|x-y\|^2/\gamma^2}$ in terms of $e^{-\|s'-t'\|^2/\gamma^2}$. Note that the difference between the means of (the distributions on) $x,y$ is $\|\mu_P - \mu_Q\|^2$ and the that the difference between the means of (the distributions on) $s',t'$ is also $\|U^\top \mu_P - U^\top \mu_Q\|^2 = \|\mu_P - \mu_Q\|^2$ since $U$ is orthogonal. So all the problem parameters remain the same, except we shift from non-independent coordinates for $x,y$ to independent coordinates for $s,t$.

\section{Proof of Lemma 1}
First note that we can rewrite the population $\MMD^2$ as
\begin{align*}
\MMD^2 &= E_{x,x' \sim P}[k(x,x')] + E_{y,y'\sim Q}[k(y,y')] - 2 E_{x \sim P,y \sim Q}[k(x,y)] \\
&= \int_u \int_v e^{-\frac{\|u - v\|^2}{\gamma^2}} p(u) p(v) dv du + \int_u \int_v e^{-\frac{\|u - v\|^2}{\gamma^2}} q(u) q(v) dv du - 2 \int_u \int_v e^{-\frac{\|u - v\|^2}{\gamma^2}} p(u) q(v) dv du\\\\
\end{align*}
We calculate each of these integrals in the following manner. Since the coordinates of the $P$ and $Q$ are independent, we have
\begin{align*}
\int_u \int_v e^{-\frac{\|u - v\|^2}{\gamma^2}} p(u) p(v) dv du &= \prod_i \left( \int_u \int_v e^{-\frac{(u - v)^2}{\gamma^2}} p_i(u) p_i(v) dv du \right) \\
&\approx \prod_i \int_u \int_v \left[1 - \frac{(u - v)^2}{\gamma^2}\right] p_i(u) p_i(v) du dv \\
&= \left( 1 - \frac{2\sigma^2}{\gamma^2} \right)^d
\end{align*}

The last two steps follow from the fact that the coordinates are independent and definition of the second moments of the distributions $p_i$ and $q_i$ (see Section~\ref{sec:doubint} of the Appendix). Similarly the corresponding term for distribution $Q$ is
\begin{align*}
\int_u \int_v e^{-\frac{\|u - v\|^2}{\gamma^2}} q(u) q(v) dv du &\approx \left( 1 - \frac{2\sigma^2}{\gamma^2} \right)^d
\end{align*}

For the final term, we have
\begin{align*}
\int_u \int_v e^{-\frac{\|u - v\|^2}{\gamma^2}} p(u) q(v) dv du
& \approx \prod_i \left( \int_u \int_v \left[ 1 - \frac{(u-v)^2}{\gamma^2}\right]p_i(u) \,q_i(v) \,du \,dv \right) \\
&=  \prod_i \left( \int_u \int_v \left( 1 - \frac{(u-\mu_{Pi})^2}{\gamma^2} - \frac{(v-\mu_{Pi})^2}{\gamma^2} + 2 \frac{(u-\mu_{Pi})(v-\mu_{Pi})}{\gamma^2} \right) p_i(u) q_i(v) du dv \right) \\
&=  \prod_i \left( \int_v \left( 1- \frac{\sigma^2}{\gamma^2} - \frac{(v-\mu_{Qi} + \mu_{Qi}-\mu_{Pi})^2}{\gamma^2} \right) q_i(v) dv \right) \\
&=  \prod_i \left( 1 - \frac{\sigma^2}{\gamma^2} - \frac{\sigma^2}{\gamma^2} - \frac{\delta_i^2}{\gamma^2} \right)
\end{align*}

The second step follows from since integral. The third step follows from independence of the coordinates. The fourth step follows from taylor expansion. The final few steps follow from the definition of second moment of the distributions (see Section~\ref{sec:doubint} of the Appendix). Combining the above terms, we have
\begin{align*}
\MMD^2 &\approx \prod_i \left( 1 - \frac{2\sigma^2}{\gamma^2} \right) + \prod_i \left( 1 - \frac{2\sigma^2}{\gamma^2} \right) - 2 \prod_i \left( 1 - \frac{\sigma^2}{\gamma^2} - \frac{\sigma^2}{\gamma^2} - \frac{\delta_i^2}{\gamma^2} \right) \\ 
&\approx 1 - \sum_i \frac{2\sigma^2}{\gamma^2} + 1 - \sum_i \frac{2\sigma^2}{\gamma^2} - 2 \left( 1 - \sum_i \frac{\sigma^2}{\gamma^2} - \sum_i \frac{\sigma^2}{\gamma^2} - \sum_i \frac{\delta_i^2}{\gamma^2}  \right) \\
&= \frac{2\|\delta\|^2}{\gamma^2}
\end{align*}


\section{Proof of Lemma 2}\label{sec:var}

The variance for the linear time MMD is given by 
\begin{align*}
var_{z,z,'}(h(z,z')) &= \E_{z,z'}[h^2(z,z')] - (\E_{z,z'}h(z,z'))^2
\end{align*}
where $h(z,z') = k(x,x') + k(y,y') - k(x,y') - k(x',y)$ where $x,x' \sim P$ and $y,y' \sim Q$ and $\E_{z,z'}[h(z,z')] = MMD^2$. Hence
the second term is just $(\E_{z,z'}h(z,z'))^2 = (MMD^2)^2$. Let us concentrate on the first term:
\begin{align*}
 \E_{z,z'}[h^2(z,z')] &= \textcolor{red}{E_{x,x'\sim P} k^2(x,x')} + \textcolor{red}{E_{y,y' \sim Q} k^2(y,y')} + \textcolor{blue}{2E_{x\sim P,y \sim Q} k^2(x,y)} \\
 &+ \textcolor{Bittersweet}{2 E_{x,x' \sim P,y,y' \sim Q} k(x,x')k(y,y')} + \textcolor{NavyBlue}{2 E_{x,x' \sim P,y,y' \sim Q} k(x,y')k(x',y) }\\
& - \textcolor{magenta}{4E_{x,x' \sim P,y \sim Q} k(x,x')k(x,y)} - \textcolor{magenta}{4E_{x \sim P,y,y' \sim Q} k(x,y)k(y,y') }
\end{align*} Hence, there are five kinds of terms to calculate
\begin{enumerate}
\item \textcolor{red}{$E_{x,x'\sim P} k^2(x,x')$} (from which \textcolor{red}{$E_{y,y'\sim Q} k^2(y,y')$} can follow)
\item \textcolor{blue}{$E_{x\sim P,y \sim Q} k^2(x,y)$}
\item \textcolor{Bittersweet}{$E_{x,x' \sim P,y,y' \sim Q} k(x,x')k(y,y')$}
\item \textcolor{NavyBlue}{$E_{x,x' \sim P,y,y' \sim Q} k(x,y')k(x',y)$}
\item \textcolor{magenta}{$E_{x,x' \sim P,y \sim Q} k(x,x')k(x,y)$} (from which \textcolor{magenta}{$E_{x \sim P,y,y' \sim Q} k(x,y)k(y,y')$} can follow)
\end{enumerate}
Let us calculate these five terms in order. 

\subsection{Term 1: \textcolor{red}{$E_{x,x'\sim P} k^2(x,x')$}}

\begin{align*}
& = \quad \int_{x,x' \sim P} e^{-2\frac{\|x-x'\|^2}{\gamma^2}} p(x) p(x') dx dx' \\
&= \quad \prod_i \left(\int_{x,x'} e^{-2\frac{(x_i-x_i')^2}{\gamma^2}} p_i(x_i) p_i(x_i') dx_i dx_i' \right) \\
&\approx \quad \prod_i \left(1 - \frac{4\sigma^2}{\gamma^2} + \frac{4\mu_{4}}{\gamma^4} + \frac{12\sigma^4}{\gamma^4}\right) \\
&\approx \quad \textcolor{red}{1 - \frac{4 d \sigma^2}{\gamma^2} + \frac{4d\mu_4}{\gamma^4} + \frac{12 d \sigma^4}{\gamma^4} + \frac{8d(d-1)\sigma^4}{\gamma^4}}
\end{align*}

The third step follows from our calculations in Section~\ref{sec:doubint} of the Appendix. Note that the extra terms arise from considering all cross terms with denominator $\gamma^4$.
\subsection{Term 2: \textcolor{blue}{$E_{x\sim P,y \sim Q} k^2(x,y)$}}
\begin{align*}
& = \quad \int_{x 
\sim P} \int_{y' \sim Q} e^{-2\frac{\|x-y'\|^2}{\gamma^2}} p(x) q(y') dx dy'\\
&= \quad \prod_i \int \int e^{-2\frac{(x_i-y_i')^2}{\gamma^2}} p_i(x_i) q_i(y_i') dx_i dy_i' \\
&\approx \quad \prod_i \Bigg(1 - \frac{4\sigma^2}{\gamma^2} - \frac{2\delta_i^2}{\gamma^2}    + \frac{4\mu_{4}}{\gamma^4}  + \frac{24\sigma^2\delta_i^2}{\gamma^4} + \frac{12\sigma^4}{\gamma^4}  + \frac{2\delta_i^4}{\gamma^4}\Bigg) \\
&\approx \quad \textcolor{blue}{ 1 - \frac{4d\sigma^2}{\gamma^2} - \frac{2\|\delta\|^2}{\gamma^2}
 + \frac{4d\mu_4}{\gamma^4} + \frac{24 \sigma^2 \|\delta\|^2}{\gamma^4} 
+ \frac{12d\sigma^4}{\gamma^4} + \frac{\cancel{2\|\delta\|_4^4}}{\gamma^4} + \frac{8d(d-1)\sigma^4}{\gamma^4} + \frac{8(d-1)\sigma^2\|\delta\|^2}{\gamma^4}+ \frac{2\|\delta\|^4 \cancel{- 2\|\delta\|_4^4}}{\gamma^4}  }
\end{align*}
The third step follows from our calculations in Section~\ref{sec:doubint} of the Appendix.

\subsection{Term 3: \textcolor{Bittersweet}{$E_{x,x' \sim P,y,y' \sim Q} k(x,x')k(y,y')$}}
\begin{align*}
& = \quad \int_{x,x' 
\sim P} \int_{y,y' \sim Q} e^{-\frac{\|x-x'\|^2}{\gamma^2}} e^{-\frac{\|y-y'\|^2}{\gamma^2}} p(x) p(x') q(y) q(y') dx dx' dy dy'\\
&=\quad  \prod_i \int \int e^{-\frac{(x_i - x_i')^2}{\gamma^2}} p_i(x_i) p_i(x_i') dx_i dx_i' \prod_i \int \int e^{-\frac{(y_i - y_i')^2}{\gamma^2}} q_i(y_i) q_i(y_i') dy_i dy_i' \\
&\approx \quad \prod_i \left( 1 - \frac{2\sigma^2}{\gamma^2} + \frac{\mu_{4}}{\gamma^4} + \frac{3\sigma^4}{\gamma^4} \right) \left( 1 - \frac{2\sigma^2}{\gamma^2} + \frac{\mu_{4}}{\gamma^4} + \frac{3\sigma^4}{\gamma^4} \right) \\
&\approx \quad \prod_i \left(1 - \frac{4\sigma^2}{\gamma^2}  + \frac{2\mu_{4}}{\gamma^4}  + \frac{10\sigma^4}{\gamma^4}   \right) \\
&\approx \quad \textcolor{Bittersweet}{1 - \frac{4d\sigma^2}{\gamma^2} + \frac{2d \mu_4}{\gamma^4} + \frac{10d\sigma^4}{\gamma^4}
 + \frac{8d(d-1)\sigma^4}{\gamma^4} }
\end{align*}
The third step follows from our calculations in Section~\ref{sec:doubint} of the Appendix.

\subsection{Term 4: \textcolor{NavyBlue}{$E_{x,x' \sim P,y,y' \sim Q} k(x,y)k(x',y')$}}
\begin{align*}
& = \quad \int_{x,x' 
\sim P} \int_{y,y' \sim Q} e^{-\frac{\|x-y\|^2}{\gamma^2}} e^{-\frac{\|x'-y'\|^2}{\gamma^2}} p(x) p(x') q(y) q(y') dx dx' dy dy'\\
&  \approx \quad \prod_i \Bigg(1 - \frac{2\sigma^2}{\gamma^2} - \frac{\delta_i^2}{\gamma^2}  + \frac{\mu_{4}}{\gamma^4} +  \frac{6\sigma^2\delta_i^2}{\gamma^4} + \frac{3\sigma^4}{\gamma^4} + \frac{\delta_i^4}{2\gamma^4} \Bigg)^2 \\
&= \quad \prod_i \left( 1  - \frac{4\sigma^2}{\gamma^2} - \frac{2\delta_i^2}{\gamma^2}  + \frac{2\mu_{4}}{\gamma^4} +  \frac{16\sigma^2\delta_i^2}{\gamma^4} + \frac{10\sigma^4}{\gamma^4}  + \frac{2\delta_i^4}{\gamma^4} \right) \\
&\approx \quad  \textcolor{NavyBlue}{1 - \frac{4 d \sigma^2}{\gamma^2} - \frac{2\|\delta\|^2}{\gamma^2} + \frac{2d\mu_4}{\gamma^4} + \frac{16\sigma^2\|\delta\|^2}{\gamma^4} + \frac{10d \sigma^4}{\gamma^4} +  \cancel{\frac{2\|\delta\|_4^4}{\gamma^4}}  + \frac{8d(d-1)\sigma^4}{\gamma^4} + \frac{2\|\delta\|^4 \cancel{- 2\|\delta\|_4^4}}{\gamma^4} + \frac{8(d-1)\sigma^2\|\delta\|^2}{\gamma^4} }
\end{align*}
The third step follows from our calculations in Section~\ref{sec:doubint} of the Appendix.

\subsection{Term 5: \textcolor{magenta}{$E_{x,x' \sim P,y \sim Q} k(x,x')k(x,y)$ }}
\begin{align*}
& = \quad \int_{x,x' \sim P} \int_{y \sim Q} e^{-\frac{\|x-x'\|^2}{\gamma^2}} e^{-\frac{\|x'-y\|^2}{\gamma^2}} p(x) p(x') q(y) dx dx' dy \\
&= \quad \prod_i \left(\int \int e^{-\frac{\|x_i-x_i'\|^2}{\gamma^2}} e^{-\frac{\|x_i'-y_i\|^2}{\gamma^2}} p_i(x_i) p_i(x_i') q(y_i) \right) \\
&\approx \quad \prod_i \Bigg(1 -  \frac{4\sigma^2}{\gamma^2}  - \frac{\delta_i^2}{\gamma^2} + 
\frac{3\mu_{4}}{\gamma^4} + \frac{8\sigma^2\delta_i^2}{\gamma^4} + \frac{9\sigma^4}{\gamma^4}  + \frac{\delta^4}{2\gamma^4} + \frac{2\mu_{3}\delta_i}{\gamma^4}\Bigg) \\
&\approx \quad \textcolor{magenta}{ 1 - \frac{4d\sigma^2}{\gamma^2} - \frac{\|\delta\|^2}{\gamma^2}  + \frac{3d \mu_4}{\gamma^4} +\frac{8\sigma^2 \|\delta\|^2}{\gamma^4} + \frac{9 d \sigma^4}{\gamma^4} + \cancel{\frac{\|\delta\|_4^4}{2\gamma^4}} + \frac{2\mu_3 \sum_i\delta_i}{\gamma^4} + \frac{8d(d-1) \sigma^4}{\gamma^4}  +  \frac{4(d-1)\sigma^2 \|\delta\|^2 }{\gamma^4} + \frac{\|\delta\|^4 \cancel{- \|\delta\|_4^4}}{2\gamma^4}}
\end{align*}

The second step follows from our calculations in Section~\ref{sec:tripleint} of the Appendix. Combining the all the terms above, we get the following bound on the variance.

\subsection{The bound on $ \E_{z,z'}[h^2(z,z')]$}

\begin{align*}
 &\approx \quad \textcolor{red}{E_{x,x'\sim P} k^2(x,x')} + \textcolor{red}{E_{y,y' \sim Q} k^2(y,y')} + \textcolor{blue}{2E_{x\sim P,y \sim Q} k^2(x,y)} \\
 & \quad \quad + \textcolor{Bittersweet}{2 E_{x,x' \sim P,y,y' \sim Q} k(x,x')k(y,y')} + \textcolor{NavyBlue}{2 E_{x,x' \sim P,y,y' \sim Q} k(x,y')k(x',y) }\\
& \quad \quad - \textcolor{magenta}{4E_{x,x' \sim P,y \sim Q} k(x,x')k(x,y)} - \textcolor{magenta}{4E_{x \sim P,y,y' \sim Q} k(x,y)k(y,y') }\\
&= \quad \textcolor{red}{\left( \cancel{1 - \frac{4 d \sigma^2}{\gamma^2}} + \cancel{\frac{4d\mu_4}{\gamma^4}} + \boxed{\frac{12 d \sigma^4}{\gamma^4}} + \cancel{\frac{8d(d-1)\sigma^4}{\gamma^4}} \right)}\\
& \quad + \textcolor{red}{\left( \cancel{1 - \frac{4 d \sigma^2}{\gamma^2}} + \cancel{\frac{4d\mu_4}{\gamma^4}} + \boxed{\frac{12 d \sigma^4}{\gamma^4}} + \cancel{\frac{8d(d-1)\sigma^4}{\gamma^4}} \right)}\\
&  \quad + \textcolor{blue}{2 \left( \cancel{1 - \frac{4d\sigma^2}{\gamma^2} - \frac{2\|\delta\|^2}{\gamma^2}}
 + \cancel{\frac{4d\mu_4}{\gamma^4}} 
+ \boxed{\frac{24  \sigma^2 \|\delta\|^2}{\gamma^4}}  +\boxed{\frac{12d\sigma^4}{\gamma^4}} + \cancel{\frac{8d(d-1)\sigma^4}{\gamma^4}} + \cancel{\frac{8(d-1)\sigma^2\|\delta\|^2}{\gamma^4}} + \boxed{\frac{2\|\delta\|^4}{\gamma^4}} \right)}\\
&  \quad + \textcolor{Bittersweet}{2\left( \cancel{1 - \frac{4d\sigma^2}{\gamma^2}} + \cancel{\frac{2d \mu_4}{\gamma^4}} + \boxed{\frac{10d\sigma^4}{\gamma^4}}
 + \cancel{\frac{8d(d-1)\sigma^4}{\gamma^4}} \right)}\\
 &   \quad + \textcolor{NavyBlue}{2\left( \cancel{1 - \frac{4 d \mu_{2}}{\gamma^2} - \frac{2\|\delta\|^2}{\gamma^2}} + \cancel{\frac{2d\mu_4}{\gamma^4}} + \boxed{\frac{16\sigma^2\|\delta\|^2}{\gamma^4}} + \boxed{\frac{10d \sigma^4}{\gamma^4}} + \cancel{\frac{8d(d-1)\sigma^4}{\gamma^4}}  + \boxed{\frac{2\|\delta\|^4}{\gamma^4}} + \cancel{\frac{8(d-1)\sigma^2\|\delta\|^2}{\gamma^4}}  \right)}\\
 &  \quad - \textcolor{magenta}{ 4\left( \cancel{1 - \frac{4d\sigma^2}{\gamma^2} - \frac{\|\delta\|^2}{\gamma^2}}  + \cancel{\frac{3d \mu_4}{\gamma^4}} +\boxed{\frac{8\sigma^2 \|\delta\|^2}{\gamma^4}} + \boxed{\frac{9 d \sigma^4}{\gamma^4}} +  \cancel{\frac{2d\mu_3 \sum_i\delta_i}{\gamma^4}} +\cancel{\frac{8d(d-1) \sigma^4}{\gamma^4}} + \cancel{\frac{4(d-1)\sigma^2 \|\delta\|^2 }{\gamma^4}} + \boxed{\frac{\|\delta\|^4}{2\gamma^4}} \right)}\\
  &  \quad - \textcolor{magenta}{ 4\left( \cancel{1 - \frac{4d\sigma^2}{\gamma^2} - \frac{\|\delta\|^2}{\gamma^2}}  + \cancel{\frac{3d \mu_4}{\gamma^4}} +\boxed{\frac{8\sigma^2 \|\delta\|^2}{\gamma^4}} + \boxed{\frac{9 d \sigma^4}{\gamma^4}} -  \cancel{\frac{2d\mu_3 \sum_i\delta_i}{\gamma^4}} +\cancel{\frac{8d(d-1) \sigma^4}{\gamma^4}} + \cancel{\frac{4(d-1)\sigma^2 \|\delta\|^2 }{\gamma^4}} + \boxed{\frac{\|\delta\|^4}{2\gamma^4}} \right)}\\
 &= \boxed{\frac{4\|\delta\|^4 }{\gamma^4} + \frac{16 d \sigma^4}{\gamma^4}  + \frac{16 \sigma^2 \|\delta\|^2}{\gamma^4}}
\end{align*} 

Finally, using the bound derived above on $\mathbb{E}_{z,z'}[h^2(z,z')]$, the bound on variance is
$$
\text{var}_{z,z,'}(h(z,z')) = \E_{z,z'}[h^2(z,z')] - (\E_{z,z'}h(z,z'))^2 = \frac{16 d \sigma^4}{\gamma^4}  + \frac{16 \sigma^2 \|\delta\|^2}{\gamma^4}. \hfill \qedhere
$$

\section{Proof of Lemma 3}
\subsection{Upper bound on $\tau_4$}
We derive the upper bound on $\tau_4$ in this section. An upper bound on $E_{z,z'}[(h(z,z') - E_{z,z'}[h(z,z')])^4]$ can be obtain in the following manner.
First note that
\begin{align*}
E_{z,z'}[(h(z,z') &- E_{z,z'}[h(z,z')])^4] = E[h^4(z,z')] - 3 (\MMD^2)^4 - 4 E[h^3(z,z')] \MMD^2 + 6 E[h^2(z,z')] (\MMD^2)^2 \\
&=  16\kappa_4 - \frac{48\|\delta\|^8}{\gamma^8} - 64 \kappa_3 \frac{\|\delta\|^2}{\gamma^2}  + \frac{96\|\delta\|^8 }{\gamma^8} + \frac{384 d \sigma^4\|\delta\|^4}{\gamma^8}  + \frac{384 \sigma^2 \|\delta\|^6}{\gamma^8} 
\end{align*}
where $\kappa_4 = E[h^4(z,z')]$ and $\kappa_3 = E[h^3(z,z')]$.

\subsection*{Calculations for $\kappa_4$}

We now calculate an upper bound to $E_{z,z'}[h^4(z,z')]$ in the following manner. With slight abuse of notation, we use $x_i$ to denote the $i^{\text{th}}$ coordinate of $x$.
We first note that
\begin{align*}
E_{z,z'}[h^4(z,z')] &= E_{z,z'} [k(x,x') + k(y,y') - k(x,y') - k(x',y)]^4 \\
&\approx E_{z,z'} \left[1 - \frac{\|x - x'\|^2}{\gamma^2} + 1 - \frac{\|y - y'\|^2}{\gamma^2} - 1 + \frac{\|x - y'\|^2}{\gamma^2} - 1 + \frac{\|x' - y'\|^2}{\gamma^2} \right]^4 \\
&= 16 E_{z,z'}\left[\frac{(x^\top x' + y^\top y' - x^\top y' - x'^\top y)}{\gamma^2} \right]^4 \\
& = 16 E_{z,z'}\left[\frac{\sum_{j=1}^d (x_j - y_j)(x'_j - y'_j)}{\gamma^2} \right]^4 \\
&= \frac{16}{\gamma^8} E_{z,z'}\left[\sum_{k_1 + \cdots + k_d = 4} {4 \choose k_1 \cdots k_d} \prod_{1 \leq i \leq d} (x_i - y_i)^{k_i}(x'_i - y'_i)^{k_i}\right] \\
&= \frac{16}{\gamma^8} \sum_{k_1 + \cdots + k_d = 4} {4 \choose k_1 \cdots k_d} \prod_{1 \leq i \leq d}(E_z[(x_i - y_i)^{k_i}])^2\\
\end{align*}

The above summation splits into five different sums, based on the different ways to write $k_1 + \cdots + k_d = 4$ - we derive these terms using the calculations in Section~\ref{sec:doubint} and Section~\ref{sec:tripleint}, as well as some terms from the Variance calculations in Section~\ref{sec:var}, and explain in brackets which way to sum the $k_i$s to 4 was used.
\begin{align*}
\kappa \quad &= \quad \color{Bittersweet}{\frac{1}{\gamma^8} \sum_i [2\mu_4 + 12\sigma^2\delta_i^2 + 6\sigma^4 + \delta_i^4]^2} \quad (\text{using (4,0,0...)}) \\
& + \quad  \color{magenta}{\frac{4}{\gamma^8} \sum_{i \neq j} (\delta_i^3 + 6\sigma^2\delta_i)^2 \delta_j^2} \quad (\text{using (3,1,0,0...)})  \\
& + \quad \color{NavyBlue}{\frac{3}{\gamma^8} \sum_{i \neq j} (4\sigma^4 + \delta_i^4 + 4 \sigma^2\delta_i^2)(4\sigma^4 + \delta_j^4 + 4\sigma^2\delta_j^2)} \quad (\text{using (2,2,0,0...)})  \\
& + \quad \color{red}{\frac{6}{\gamma^8} \sum_{i \neq j \neq k} (4\sigma^4 + \delta_i^4 + 4\sigma^2\delta_i^2) \delta_j^2 \delta_k^2} \quad (\text{using (2,1,1,0,0...)})\\
&+ \quad \color{blue}{ \frac{1}{\gamma^8} \sum_{i \neq j \neq k \neq l} \delta_i^2 \delta_j^2 \delta_k^2 \delta_l^2} \quad (\text{using (1,1,1,1,0,0...)})
\end{align*}

Expanding the each of the above terms further, we get

\begin{align*}
\text{\color{Bittersweet}{Term 1:}} \quad \quad &\frac1{\gamma^8} \Bigg[4d\mu_4^2 + (144+12) \sigma^4 \sum_i \delta_i^4 + 36 \sigma^8 d + \sum_i \delta_i^8 \\
& + 48 \mu_4 \sigma^2 \|\delta\|^2 + 24 d \mu_4 \sigma^4 + 4 \mu_4 \sum_i \delta_i^4 + 144 \sigma^6 \|\delta\|^2 + 24 \sigma^2 \sum_i \delta_i^6 \Bigg] \\
 \text{\color{magenta}{Term 2:}} \quad \quad & \frac{4}{\gamma^8} \Bigg[ \sum_{i\neq j} \delta_i^6 \delta_j^2 + 36 \sigma^4 \sum_{i \neq j} \delta_i^2 \delta_j^2 + 12\sigma^2 \sum_{i \neq j} \delta_i^4 \delta_j^2 \Bigg]\\
\text{\color{NavyBlue}{Term 3:}} \quad \quad & \frac{3}{\gamma^8} \Bigg[ 8d(d-1) \sigma^8 + \sum_{i\neq j} \delta_i^4 \delta_j^4 + 8 \sigma^4 (d-1) \sum_i \delta_i^4 + 32 \sigma^6 \|\delta\|^2 (d-1) + 8 \sigma^2 \sum_{i \neq j} \delta_i^4 \delta_j^2 + 16\sigma^4 \sum_{i \neq j} \delta_i^2 \delta_j^2 \Bigg]\\
\text{\color{red}{Term 4:}} \quad \quad & \frac{6}{\gamma^8} \Bigg[ 4\sigma^4 (d-2) \sum_{i \neq j}\delta_i^2 \delta_j^2 +  \sum_{i \neq j \neq k} \delta_i^4 \delta_j^2 \delta_k^2 + 4 \sigma^2 \sum_{i \neq j \neq k} \delta_i^2 \delta_j^2 \delta_k^2 \Bigg] \\
\text{\color{blue}{Term 5:}} \quad \quad  & \frac{1}{\gamma^8} \Bigg[ \sum_{i \neq j \neq k \neq l} \delta_i^2 \delta_j^2 \delta_k^2 \delta_l^2 \Bigg] \\
\end{align*}


\subsection*{Calculations for $\kappa_3$}

Similar to the multinomial expansion for $\kappa_4$, we have
\begin{align*}
\kappa_3 \quad &= \quad  \color{magenta}{\frac{1}{\gamma^6} \sum_{i} (\delta_i^6 + 36\sigma^4\delta_i^2 + 12\sigma^2\delta_i^4)} \quad (\text{using (3,0,0,0...)})  \\
& + \quad \color{red}{\frac{3}{\gamma^6} \sum_{i \neq j} (4\sigma^4 + \delta_i^4 + 4\sigma^2\delta_i^2) \delta_j^2} \quad (\text{using (2,1,0,0...)})\\
&+ \quad \color{blue}{ \frac{1}{\gamma^6} \sum_{i \neq j \neq k} \delta_i^2 \delta_j^2 \delta_k^2} \quad (\text{using (1,1,1,0,0...)})
\end{align*}

Using the above expansion, we get
\begin{align*}
\kappa_3 \frac{\|\delta\|^2}{\gamma^2} \quad &= \quad  \color{magenta}{\frac{1}{\gamma^8} \sum_{i \neq j} (\delta_i^6 \delta_j^2 + 36\sigma^4\delta_i^2 \delta_j^2 + 12\sigma^2\delta_i^4 \delta_j^2)} \\
& + \quad  \color{magenta}{\frac{1}{\gamma^8} \sum_{i} (\delta_i^8 + 36\sigma^4\delta_i^4 + 12\sigma^2\delta_i^6)} \\
& + \quad \color{red}{\frac{3}{\gamma^8} \sum_{i \neq j \neq k} (4\sigma^4 \delta_j^2 \delta_k^2 + \delta_i^4 \delta_j^2 \delta_k^2 + 4\sigma^2\delta_i^2\delta_j^2 \delta_k^2)} \\
& + \quad \color{red}{\frac{3}{\gamma^8} \sum_{i \neq j} (4\sigma^4 \delta_i^2 \delta_j^2 + \delta_i^6 \delta_j^2 + 4\sigma^2\delta_i^4\delta_j^2)} \\
& + \quad \color{red}{\frac{3}{\gamma^8} \sum_{i \neq j} (4\sigma^4 \delta_j^4 + \delta_i^4 \delta_j^4 + 4\sigma^2\delta_i^2\delta_j^4)} \\
&+ \quad \color{blue}{ \frac{1}{\gamma^8} \sum_{i \neq j \neq k \neq l} \delta_i^2 \delta_j^2 \delta_k^2 \delta_l^2 +\frac{3}{\gamma^8} \sum_{i \neq j \neq k} \delta_i^4 \delta_j^2 \delta_k^2}
\end{align*}


Also note the following expansions of $\|\delta\|^8$ and $\|\delta\|^6$.
\begin{align*}
\frac{\|\delta\|^8}{\gamma^8} &= \sum_{i} \delta_i^8 + 4 \sum_{i \neq j} \delta_i^6 \delta_j^2 + 3 \sum_{i \neq j} \delta_i^4 \delta_j^4 + 6 \sum_{i \neq j \neq k} \delta_i^4 \delta_j^2 \delta_k^2 + \sum_{i \neq j \neq k \neq l} \delta_i^2 \delta_j^2 \delta_k^2 \delta_l^2 \\
\frac{\|\delta\|^6}{\gamma^6} &= \sum_{i} \delta_i^6 + 3 \sum_{i \neq j} \delta_i^4 \delta_j^2 +  \sum_{i \neq j \neq k} \delta_i^2 \delta_j^2 \delta_k^2
\end{align*}

\subsection*{Putting all terms together}

Using the above calculations for $\kappa_3$ and $\kappa_4$, we obtain the following bound on $E_{z,z'}[(h(z,z') - E_{z,z'}[h(z,z')])^4]$.

\begin{align*}
E_{z,z'}[(h(z,z') &- E_{z,z'}[h(z,z')])^4] = E[h^4(z,z')] - 3 (\MMD^2)^4 - 4 E[h^3(z,z')] \MMD^2 + 6 E[h^2(z,z')] (\MMD^2)^2 \\
&=  16\kappa_4 - \frac{48\|\delta\|^8}{\gamma^8} - 64 \kappa_3 \frac{\|\delta\|^2}{\gamma^2}  + \frac{96\|\delta\|^8 }{\gamma^8} + \frac{384 d \sigma^4\|\delta\|^4}{\gamma^8}  + \frac{384 \sigma^2 \|\delta\|^6}{\gamma^8} \\
&= \frac{16}{\gamma^2} \Bigg(4d\mu_4^2 + 36\sigma^8 d + 24d\mu_4\sigma^4 + 24d(d-1)\sigma^8 + (96d \sigma^6 + 48\sigma^6 +  48 \mu_4 \sigma^2) \sum_i \delta_i^2 \\
& \quad \quad \quad \quad + (132 \sigma^4 + 4 \mu_4) \sum_i \delta_i^4 + 144 \sigma^4 \sum_{i \neq j} \delta_i^2 \delta_j^2\Bigg) - 64 \left(24 \sigma^4 \sum_i \delta_i^4 + 24 \sigma^4 \sum_{i \neq j} \delta_i^2 \delta_j^2 \right) \\
&= \frac{1}{\gamma^8} \Bigg(64d\mu_4^2 + 576\sigma^8 d + 384d\mu_4\sigma^4 + 384d(d-1)\sigma^8 + (1536d \sigma^6 + 768\sigma^6 +  768 \mu_4 \sigma^2) \sum_i \delta_i^2 \\
& \quad \quad \quad + (576 \sigma^4 + 64 \mu_4) \sum_i \delta_i^4 + 768 \sigma^4 \sum_{i \neq j} \delta_i^2 \delta_j^2 \Bigg) \\
& = (3 + o(1)) V^2
\end{align*}

where we substituted $\kappa_4,\kappa_3$ in the third equationand the $\|\delta\|^6$ and $\|\delta\|^8$ terms perfectly cancel out.

\section{Helpful Calculations for Lemma 1, 2, 3, 4}

\subsection{Double Integrals}
\label{sec:doubint}

\begin{align*}
\int_u \int_v e^{-\frac{(u - v)^2}{\gamma^2}} f(u) g(v) du dv &\approx \int_u \int_v \left[ 1 - \frac{(u - v)^2}{\gamma^2} + \frac{(u - v)^4}{2\gamma^4} \right] f(u) g(v) du dv \\
&= 1 - \frac{2\sigma^2}{\gamma^2}  - \frac{\delta^2}{\gamma^2} + \frac{\mu_4}{\gamma^4} + \frac{6\sigma^2\delta^2}{\gamma^4} + \frac{3\sigma^4}{\gamma^4} + \frac{\delta^4}{2\gamma^4}
\end{align*}

because $\int \int \frac{(u - v)^2}{\gamma^2} f(u) g(v) du dv$
\begin{align*}
 &= \int \int \frac{((u - \mu_{f}) - (v - \mu_{f}))^2}{\gamma^2} f(u) g(v) du dv \\
&= \int \left( \frac{\sigma^2}{\gamma^2} + \frac{(v - \mu_{f})^2}{\gamma^2} \right) g(v) dv = \frac{2\sigma^2}{\gamma^2} + \frac{\delta^2}{\gamma^2}\\
\end{align*}
and $ \int_u \int_v \frac{(u - v)^4}{\gamma^4} f(u) g(v) dv du$
\begin{align*}
 &= \int_u \int_v \frac{((u - \mu_{f}) - (v - \mu_{f}))^4}{\gamma^4} f(u) g(v) dv du \\
&= \int_u \int_v \left( \frac{(u - \mu_{f})^4}{\gamma^4} + \frac{(v - \mu_{f})^4}{\gamma^4} - \frac{4(u - \mu_{f})^3(v - \mu_{f})}{\gamma^4} - \frac{4(u - \mu_{f})(v - \mu_{f})^3}{\gamma^4}  + \frac{6(u - \mu_{f})^2(v - \mu_{f})^2}{\gamma^4} \right) f(u) g(v) du dv \\
&= \int_v \left(  \frac{\mu_4}{\gamma^4} + \frac{(v - \mu_{f})^4}{\gamma^4} - \frac{4\mu_3(v - \mu_{f})}{\gamma^4} + \frac{6\sigma^2(v - \mu_{f})^2}{\gamma^4} \right) g(v) dv \\
&= \frac{\mu_4}{\gamma^4} + \left[ \frac{\mu_4}{\gamma^4} - \frac{4\mu_3\delta}{\gamma^4} + \frac{6\sigma^2\delta^2}{\gamma^4}  + \frac{\delta^4}{\gamma^4} \right]_1 + \left[ \frac{4\mu_3\delta}{\gamma^4} \right]_2 + \left[ \frac{6\sigma^2 (\sigma^2 + \delta^2)}{\gamma^4} \right]_3\\
&= \frac{2\mu_4}{\gamma^4}  + \frac{12\sigma^2\delta^2}{\gamma^4} + \frac{6\sigma^4}{\gamma^4} + \frac{\delta^4}{\gamma^4}
\end{align*}

Finally, we have
\begin{align*}
 \int_u \int_v (u - v)^3 f(u) g(v) dv du &= \int_u \int_v ((u - \mu_f) - (v - \mu_f))^3 f(u) g(v) du dv \\
 &= \int_v \mu_3 - 3\sigma^2(v - \mu_f) - (v - \mu_f)^3 g(v) dv \\
 &= 3\sigma^2 \delta + \delta^3 + 3 \sigma^2 \delta = \delta^3 + 6\sigma^2 \delta.
\end{align*}

\subsection{Triple Integral}
\label{sec:tripleint}

\begin{align*}
& \int_u \int_v \int_y e^{-\frac{(u - v)^2}{\gamma^2}} e^{-\frac{(v - y)^2}{\gamma^2}} f(u) g(v) g(y) du dv dy \\
&= \int_u \int_v \int_y \left[1 - \frac{(u - v)^2}{\gamma^2} + \frac{(u - v)^4}{2 \gamma^4}\right] \left[1 - \frac{(v - y)^2}{\gamma^2} + \frac{(v - y)^4}{2 \gamma^4}\right] f(u) g(v) g(y) du dv dy \\
&= \int_u \int_v \int_y \left[1 - \frac{(u - v)^2}{\gamma^2} - \frac{(v - y)^2}{\gamma^2} + \frac{(u - v)^4}{2 \gamma^4} + \frac{(v - y)^4}{2 \gamma^4} + \frac{(u - v)^2(v - y)^2}{\gamma^4} \right] f(u) g(v) g(y) du dv dy \\
&= 1 -  \frac{2\sigma^2}{\gamma^2} - \frac{\delta^2}{\gamma^2} - \frac{2\sigma^2}{\gamma^2} + 
\frac{1}{2} \Bigg[\frac{2\mu_4}{\gamma^4} + \frac{12\sigma^2\delta^2}{\gamma^4} + \frac{6\sigma^4}{\gamma^4} + \frac{\delta^4}{\gamma^4} \Bigg] + \frac{1}{2} \left[ \frac{2\mu_4}{\gamma^4} + \frac{6\sigma^4}{\gamma^4}\right] \\
&\quad + \frac{1}{\gamma^4} \int_v \left[\sigma^2 + (v - \mu_{f})^2\right] \left[\sigma^2 + (v - \mu_{g})^2\right] g(v) dv \\
&= 1 -  \frac{2\sigma^2}{\gamma^2} - \frac{\delta^2}{\gamma^2} - \frac{2\sigma^2}{\gamma^2} + 
\frac{1}{2} \Bigg[\frac{2\mu_4}{\gamma^4} + \frac{12\sigma^2\delta^2}{\gamma^4} + \frac{6\sigma^4}{\gamma^4} + \frac{\delta^4}{\gamma^4} \Bigg] + \frac{1}{2} \left[ \frac{2\mu_4}{\gamma^4} + \frac{6\sigma^4}{\gamma^4}\right] \\
&\quad + \frac{3\sigma^4}{\gamma^4} + \frac{2\sigma^2\delta^2}{\gamma^4} + \frac{\mu_4}{\gamma^4} - \frac{2\mu_3\delta}{\gamma^4} \\
&= 1 -  \frac{4\sigma^2}{\gamma^2} - \frac{\delta^2}{\gamma^2} + 
\frac{3\mu_4}{\gamma^4} + \frac{8\sigma^2\delta^2}{\gamma^4} + \frac{9\sigma^4}{\gamma^4} + \frac{\delta^4}{2\gamma^4} - \frac{2\mu_3\delta}{\gamma^4}
\end{align*}

The last equality is obtained from the following:
\begin{align*}
\int_v \left[\sigma^2 + (v - \mu_{f})^2\right] \left[\sigma^2 + (v - \mu_{g})^2\right] g(v) dv &= \sigma^4
+ \sigma^2(\sigma^2 + (\mu_{g} - \mu_{f})^2)  + \sigma^4 + \int_v (v - \mu_{f})^2(v - \mu_{g})^2 g(v) dv
\end{align*}

\newpage

\section{Additional Experiments}
\begin{figure} [h!]
\centering
\includegraphics[width=0.4\linewidth]{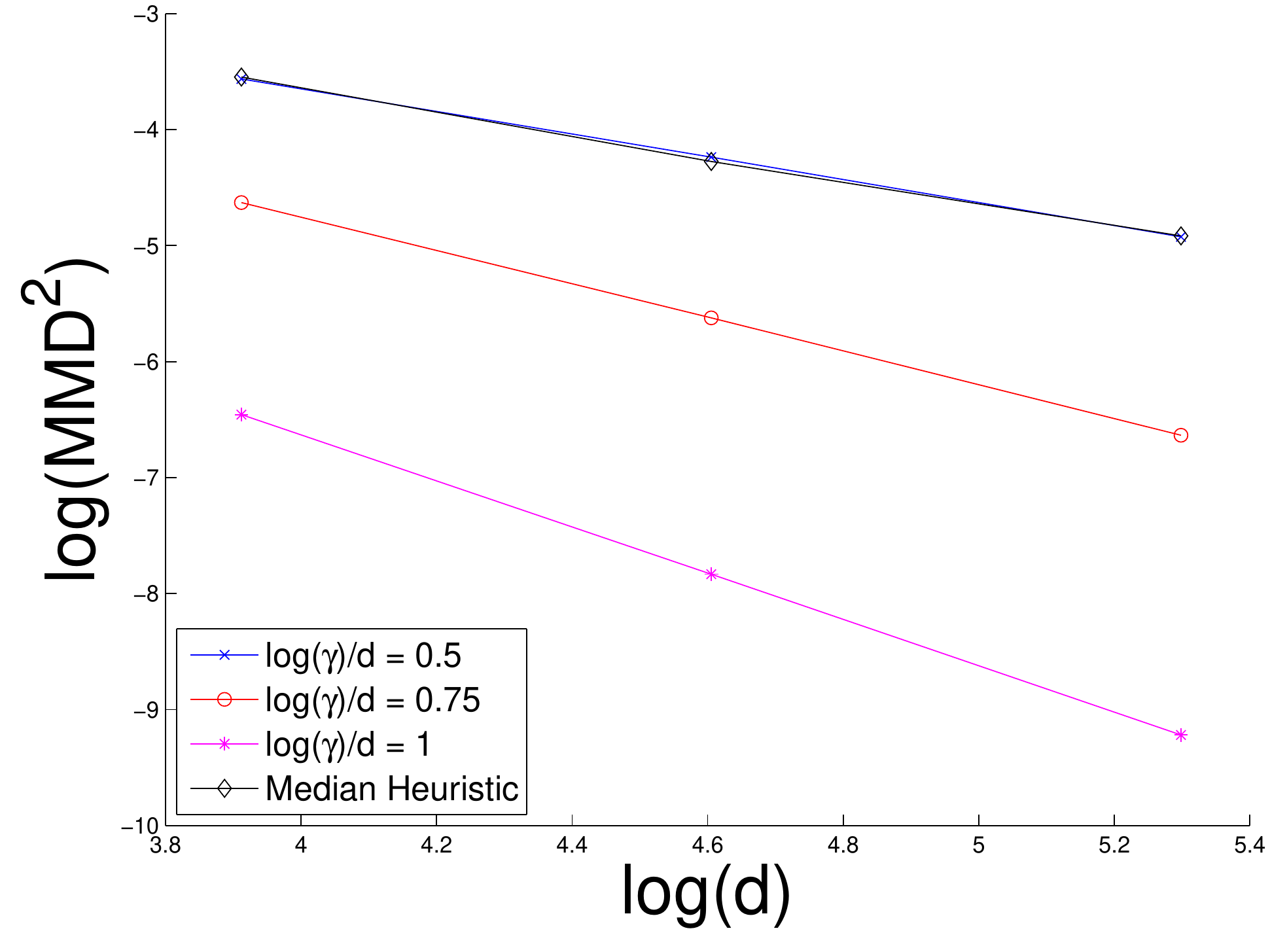}
\includegraphics[width=0.4\linewidth]{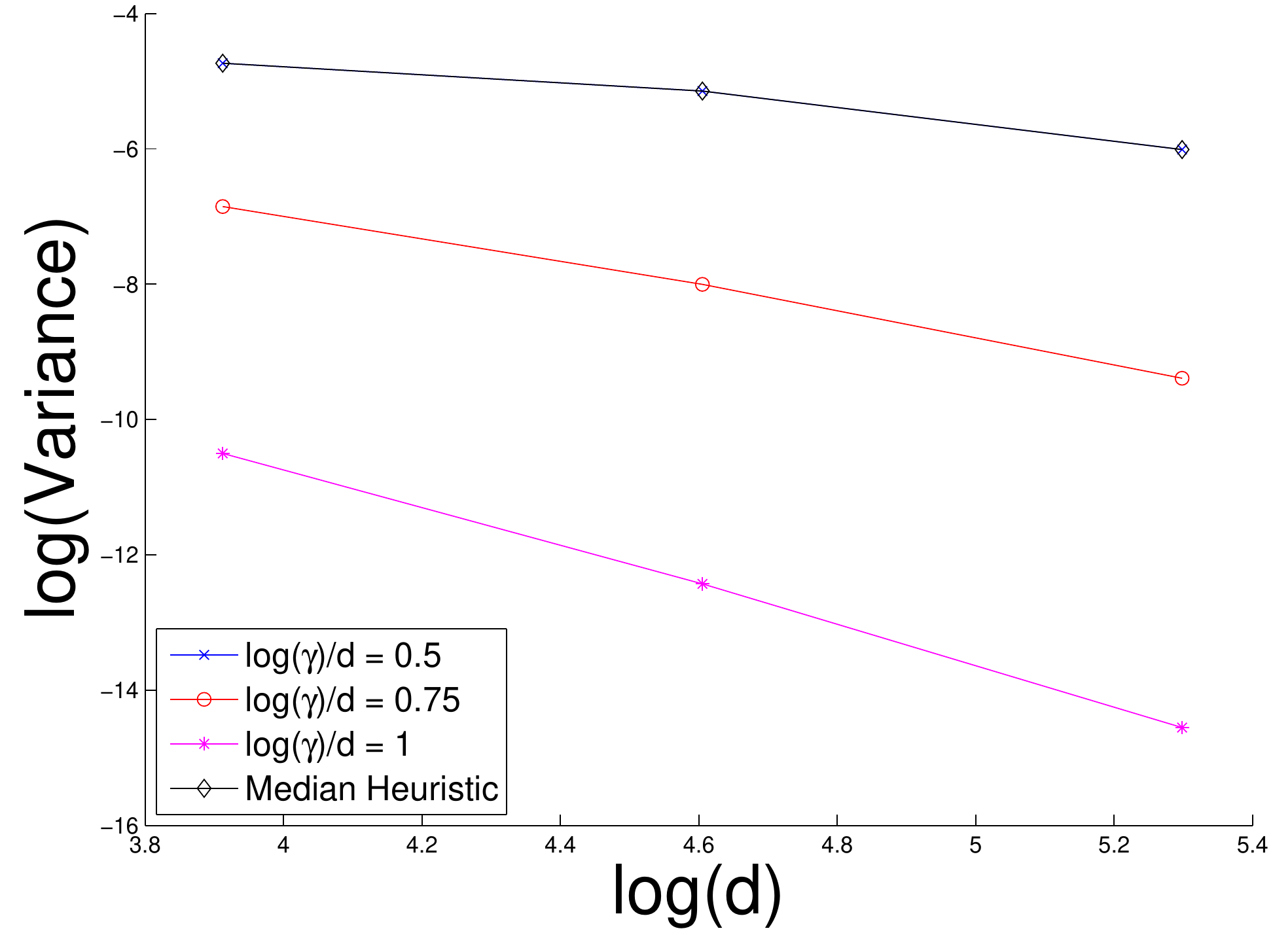}
\caption{A plot for $\MMD^2$ and Variance of linear statistic, when $n=1000$ for Normal distribution with identity covariance and $\Psi = 1$, for bandwidths $d^{0.5},d,d^{0.75}$. Note that these plots provide empirical verification for  Lemma 1 and Lemma 2}.
\label{fig:mmd-variance}
\end{figure}
